\documentclass[11pt]{article}

\usepackage[T1]{fontenc}
\usepackage{amsmath,amssymb,amsthm,mathtools}
\usepackage[a4paper,margin=1in]{geometry}
\usepackage[hidelinks]{hyperref}
\usepackage{microtype}
\usepackage{xcolor}

\usepackage[normalem]{ulem}
\usepackage{authblk}

\allowdisplaybreaks[4]
\numberwithin{equation}{section}

\theoremstyle{definition}

\newtheorem {theorem}{Theorem}[section]
\newtheorem {lemma}[theorem]{Lemma}
\newtheorem {corollary}[theorem]{Corollary}
\newtheorem {conjecture}[theorem]{Conjecture}
\newtheorem {proposition}[theorem]{Proposition}

\newtheorem {claim}{Claim}[section]

\newtheorem{definition}[theorem]{Definition}

\newcommand{\bl}{\color{blue}}

\newcommand{\cH}{\mathcal H}
\newcommand{\cG}{\mathcal G}
\newcommand{\cI}{\mathcal I}

\title{The upper bounds on occupancy fraction and variance fraction for linear regular uniform hypergraphs}

\author[1,2\thanks{ Email: weiyuanzhang@nuaa.edu.cn}]
{Weiyuan Zhang}

\author[1,2\thanks{Corresponding author. Email: kexxu1221@126.com}]
{Kexiang Xu}

\affil[1]{\footnotesize School of Mathematics,
	Nanjing University of Aeronautics and Astronautics,
	Nanjing, Jiangsu, 210016, PR China}

\affil[2]
{\footnotesize
MIIT Key Laboratory of Mathematical Modelling and High Performance Computing of Air  Vehicles, Nanjing, Jiangsu, 210016, PR China
}

\date{}

\begin{document}

\maketitle

\begin{abstract}
The hard-core model can be used to understand the number of independent sets in graphs in extremal graph theory.
The occupancy and variance fractions are the expectation and variance, respectively, of the independent-set size, normalized by the number of vertices.
We provide upper bounds on the occupancy and variance fractions of linear $d$-regular $r$-uniform hypergraphs  for small positive fugacities.
In proving these results, we use the cluster expansion and the abstract polymer models, two tools from statistical physics, to analyse auxiliary functions associated with independence polynomials, comparing their low-order coefficients and controlling the remainders uniformly in the number of vertices.

\end{abstract}


\par\vspace{1mm}

\noindent{\bfseries  AMS Classification (2020): 05C31, 05C35, 05C65, 05C69}

\section{Introduction}


Let $G$ be a graph with vertex set $V(G)$ and edge set $E(G)$.
The order of $G$ is the number of its vertices.
For a vertex $v\in V(G)$, we define its open neighborhood as $N_G(v)= \{u\in V(G):uv\in E(G)\}$ and its closed neighborhood as $N_G[v] = N_G(v)\cup\{v\}$.
The degree $d_G(u)$ of vertex $u$ is the cardinality of $N_G(u)$.
A set $S\subseteq V(G)$ in a graph $G$ is an \textit{independent set} if $G[S]$, the induced subgraph of $G$ by $S$, consists of $|S|$ isolated vertices. The \textit{independence number} $\alpha(G)$ of $G$ is the cardinality of a maximum independent set of $G$.

All hypergraphs in this paper have a nonempty vertex set and contain no empty hyperedge.
They are finite and simple, so  multiple hyperedges are not allowed.
For a hypergraph $\mathcal G$, the degree of a vertex is the number of hyperedges containing it  in $\mathcal G$.
A hypergraph is $r$-uniform if every hyperedge has size $r$,
$d$-regular if every vertex belongs to exactly $d$ hyperedges,
and linear if any two distinct hyperedges intersect in at most one vertex.
For a hyperedge $e$ of a hypergraph, we use $v(e)$ to denote the set of vertices covered by the hyperedge.
A \emph{weakly independent set} in a hypergraph is a set of vertices that contains no hyperedge.
For convenience, in the remainder of this paper we use ``independent set'' to denote ``weakly independent set'' of the hypergraph.
For either a graph or a hypergraph $\mathcal G$, let $\cI(\mathcal G)$ denote its family of independent sets, including the empty set, and put $\operatorname{ind}(\mathcal G)=|\cI(\mathcal G)|$.

The hard-core model (or hard-core lattice gas), commonly studied on the integer lattice $\mathbb{Z}^d$, is a simple lattice analogue of a hard-sphere gas. The glassy behavior of densely packed hard spheres has attracted considerable attention because of its important physical implications.
In extremal graph theory, the hard-core model can be used to understand the minimum and maximum number of independent sets \cite{DJPR2017}.
From the statistical physics perspective, counting independent sets of a graph is equivalent to counting the configurations of a hard-core lattice gas \cite{SS2005}.
Each vertex represents a site that is either empty or occupied by one particle, and the occupied vertex set contains no edge.

A basic extremal problem asks for the largest or smallest number of independent sets in a graph with prescribed parameters.
Kahn \cite{K2001}  stated that the disjoint union of $n/(2d)$ copies of the complete bipartite graph $K_{d,d}$ has the maximum number of independent sets among all $d$-regular bipartite graphs on $n$ vertices with $n$ divisible by $2d$.
Zhao \cite{Z2010} extended  the above result from $d$-regular bipartite graphs to all $d$-regular graphs.
In 2019, Sah \textit{et al.} \cite{SSSZ2019} determined the maximum number of independent sets in an irregular graph and extended their result to an upper bound on the independence polynomial.
In the other direction, Cutler and Radcliffe~\cite{CR2014} obtained the sharp lower bound on $\operatorname{ind}(G)$ for $d$-regular graphs $G$, with disjoint unions of  complete graphs $K_{d+1}$ attaining equality when the orders permit.

As for hypergraphs, Cooper \textit{et al.}~\cite{CDM2014} proved lower bounds on $\operatorname{ind}(\mathcal G)$ in linear uniform hypergraphs $\mathcal G$ with given order and average degree.
Balogh, Morris and Samotij \cite{BMS2015} and, independently, Saxton and Thomason~\cite{ST2015} developed hypergraph container theorems, providing general tools for counting independent sets.
Later, Cohen \textit{et al.}~\cite{CPST2022} obtained upper bounds on $\operatorname{ind}(\mathcal G)$  for linear regular uniform hypergraphs $\mathcal{G}$ without cross-edges  where a cross-hyperedge in the neighborhood of a vertex $v$ is a  hyperedge $e$ that contains two neighbors $u_1$, $u_2$ of $v$ but not $v$ itself.
Meanwhile, Balogh \textit{et al.} \cite{BBN2021} provided the following conjecture,  proving that it holds when the hypergraph is quasi-bipartite.
 Extremal bounds in this setting have a long history with several surprising applications. For further results beyond those discussed above, see \cite{A1996,CR2013,S1995}.

\begin{conjecture}\cite{BBN2021}
For all $r\geq 2$ and $d\in \mathbb{N}$, if $\mathcal{G}$ is a $d$-regular $r$-uniform hypergraph on $n$ vertices, then
\[ \operatorname{ind}(\mathcal G) \le \operatorname{ind}(\mathcal H_d^r)^{n/(rd)} \]  where $\mathcal H_d^r$ is a hypergraph defined in the following.
\end{conjecture}

For $r \ge 2$ and $d \in \mathbb{N}$, let $\mathcal{H}_d^r$ be the $d$-regular $r$-partite $r$-uniform hypergraph on $rd$ vertices and $d^2$ hyperedges defined as follows: mark a subset of cardinality $d$ from the vertex set, partition the remaining $(r-1)d$ vertices into $d$ sets of $r-1$ vertices each, 
 and then include in the  hyperedge set of $\mathcal{H}_d^r$ where each $r$-set consisting of a marked vertex and an $(r-1)$-subset in the partition.
For example, $\mathcal{H}_d^2$ is the complete bipartite graph $K_{d,d}$, and $\mathcal{H}_d^3$ is the set of triple subsets on $3d$ vertices where $d$  marked vertices are each joined into each $2$-set, forming a $3$-set, of the balanced partition of  the remaining $2d$ vertices into $d$ subsets.

For a hypergraph $\mathcal G$ on $n$ vertices, define its \textit{independence polynomial} by
$$
Z_{\mathcal G}(\lambda)
=\sum_{I\in\cI(\mathcal G)}\lambda^{|I|}
=\sum_{k=0}^{n}i_k(\mathcal G)\lambda^k,
$$
where $i_k(\mathcal G)$ is the number of independent sets in $\mathcal G$ of size $k$.
 There is also a natural statistical-physics viewpoint on this.
Independence polynomial is the partition functions of the hard-core model.
Thus $Z_{\mathcal G}(1)=\operatorname{ind}(\mathcal G)$.
For $\lambda>0$, the \textit{hard-core model} on $\mathcal G$ is the measure $\mu_{\mathcal G,\lambda}$: $\cI(\mathcal G)\rightarrow [0,1]$ given by
$$
\mu_{\mathcal G,\lambda}(I)
=\frac{\lambda^{|I|}}{Z_{\mathcal G}(\lambda)},
\qquad I\in\cI(\mathcal G).
$$
The parameter $\lambda$ is called the fugacity.
 This distribution is known as the grand canonical ensemble and represents particles of a gas (occupied vertices of the independent set) in a volume in thermal equilibrium within a much larger volume (and so particles can enter and leave the small volume) \cite{DJP2021}.
At $\lambda=1$, this is the uniform distribution on independent sets.
Since the empty set is independent, $Z_{\mathcal G}(\lambda)\ge1$ for $\lambda\ge0$.
All logarithms below are natural unless a base is explicitly indicated.

The \textit{free energy density}, \textit{occupancy fraction} and \textit{variance fraction} are defined \cite{DJPR2017,DST2025} as
\begin{align}
F_{\mathcal G}(\lambda)
&=\frac1n\log Z_{\mathcal G}(\lambda),\label{eq:free-energy}\\
E_{\mathcal G}(\lambda)
&=\lambda \frac{d}{d\lambda}F_\mathcal G(\lambda), \label{eq:occupancy}\\
V_{\mathcal G}(\lambda)
&=\lambda \frac{d}{d\lambda}E_\mathcal G(\lambda) .\label{eq:variance}
\end{align}

Note that
\[
n E_\mathcal G(\lambda) = \mathbb{E}_{I\sim \mu_{\mathcal G,\lambda}} (|I|) = \frac{\lambda Z_{\mathcal G}^{\prime}(\lambda)}{Z_{\mathcal G}(\lambda)}
\]
is the expected number of vertices of $\mathcal G$ occupied by a random independent set from the hard-core model.
Similarly, we know that
\[
n V_{\mathcal G}(\lambda) = \mathrm{Var}_{I\sim \mu_{\mathcal G,\lambda}}(|I|) = \mathbb{E}_{I\sim \mu_{\mathcal G,\lambda}} (|I|^2) - \left(\mathbb{E}_{I\sim \mu_{\mathcal G,\lambda}} (|I|)\right)^2
\]
is the variance of the size of an independent set from the model.
From a statistical-mechanical perspective,
$F_{\mathcal G}(\lambda)$ is also called \textit{free energy per site} or \textit{reduced pressure} \cite{SS2005}.
The occupancy fraction $E_{\mathcal G}(\lambda)$ is the average particle density, while
$V_{\mathcal G}(\lambda)$ measures particle-number fluctuations per vertex.
This is the grand canonical fluctuation--response relation \cite{LPV1967,WET2024}.
If there is no confusion, we usually use $\mathbb{E}$ to represent $\mathbb{E}_{I\sim \mu_{\mathcal G,\lambda}}$, and $\mathrm{ Var}$ to represent $\mathrm{ Var}_{I\sim \mu_{\mathcal G,\lambda}}$.
The results of Cutler, Kahn and Zhao~\cite{CR2014, K2001, Z2010} can be expressed as
$F_{K_{d+1}}(\lambda) \le F_G(\lambda) \le F_{K_{d,d}}(\lambda) $
for every $d$-regular graph $G$ and every $\lambda>0$.

It is natural to consider the generalizations of the above results to the bounds on the parameters $E_\mathcal G(\lambda)$ and $V_\mathcal G(\lambda)$, respectively.
Davies \textit{et al.} \cite{DJPR2017} introduced a novel technique, called the occupancy method, and used it to prove the tight upper bounds on the independence polynomial and the matching polynomial of $d$-regular graphs.
In the subsequent work \cite{DJPR2018}, they considered the lower bounds on $E_G(\lambda)$ for triangle-free graphs with a given maximum degree, and derived the corresponding lower bounds on the number of independent sets using the properties of the occupancy fraction.
Cutler and Radcliffe \cite{CR2018} used the same method to obtain lower bounds on $E_G(\lambda)$ for $d$-regular graphs and, as a consequence, a tight lower bound on the independence polynomial for this class of graphs.
Perarnau and Perkins \cite{PP2018} obtained an upper bound on $E_G(\lambda)$ for cubic graphs of girth at least $5$ and a lower bound on $E_G(\lambda)$ for triangle-free cubic graphs.
Davies \textit{et al.}~\cite{DST2025} studied lower bounds on occupancy fraction with a given degree-sequence and began a systematic study of extremal variance fractions. They provided the lower and upper bounds on $V_G(\lambda)$ for $0<\lambda\le \frac{1}{2n-1}$ and conjectured that the lower bound holds for all $\lambda>0$, with $n$ replaced by $\Delta+1$  for the maximum degree $\Delta$.
Zhang and Xu \cite{ZX2026} proved this conjecture and provided the lower bound on $V_G(\lambda)$ for triangle-free graphs with maximum degree $\Delta$ for all $\lambda>0$.
Davies \textit{et al.} \cite{DSST2026} subsequently extended the degree-sequence lower bounds \cite{DST2025} to a multivariate version with  every vertex fugacity less than $\frac 1\Delta$.

These parameters also have uses beyond extremal counting.
Since $\alpha(G)\ge nE_G(\lambda)$, the occupancy fraction gives lower bounds on the independence number and hence upper bounds on  off-diagonal Ramsey numbers \cite{MS2025}.
This connects the work of Shearer~\cite{S1995} with the local occupancy approach of Davies, Jenssen, Perkins, and Roberts~\cite{DJPR2018}.
The conditional independence properties of the hard-core model can be exploited in probabilistic and algorithmic methods for graph coloring and related problems \cite{BLM2022,DVKP2020,DVKP2021,DP2023,JPP2023,M2019,SS2005,Z2017}.

Before presenting our main results, we first define some parameters that will be used.
To state explicit fugacity ranges, we always set, for $r\ge3$ and $d\ge2$,
\begin{equation}\label{eq:A0-rho}
A_0=r(d-1)+1,
\qquad
\zeta=(32A_0)^{-d/r},
\qquad
M=\frac{3d}{20rA_0},
\qquad
\lambda_0(r,d)=\frac{\zeta}{2-\zeta}.
\end{equation}


\begin{theorem}\label{thm:variance-main}
Let $r\ge3$ and $d\ge2$.
If $\mathcal G$ is a finite linear, $d$-regular, $r$-uniform hypergraph, then
\begin{align}\label{eq:variance-refined-bound}
V_{\mathcal G}(\lambda) \le{}& (1-p)\Bigl(p+{d\choose 2}(2r-1)
\bigl(2r-1-2rp\bigr)p^{2r-1}-\bigl(r-(r+1)p\bigr)dp^{r}
-\bigl(2r-(2r+1)p\bigr)dA_0p^{2r} \Bigr) \notag\\& +O(p^{2r+1}),
\end{align}
for every $0<\lambda\le\lambda_0(r,d)$
with $p=\lambda/(1+\lambda)$.
\end{theorem}

\begin{theorem}\label{thm:occupancy-main}
Let $r\ge3$ and $d\ge2$.
If $\mathcal G$ is a finite linear, $d$-regular, $r$-uniform hypergraph, then
\begin{equation}\label{eq:occupancy-refined-bound}
E_{\mathcal G}(\lambda) \le p+d(1-p)\left(-p^r+\frac{(2r-1)(d-1)}{2}p^{2r-1}-A_0p^{2r}\right)+O(p^{2r+1}),
\end{equation}
for every $0<\lambda\le\lambda_0(r,d)$
with $p=\lambda/(1+\lambda)$.
\end{theorem}

Let $f$ and $h$ be functions defined for all sufficiently small $p>0$. We write $f(p)=O(h(p))$ as $p\to0^+$ if there exist constants $C>0$ and $p_0>0$ such that $|f(p)|\le C|h(p)|$ whenever $0<p\le p_0$.
The notation $O_{r,d}(\cdot)$ indicates that the constants $C$ and $p_0$ may depend only on $r$ and $d$.
Similarly, in $O_{\mathcal G}(\cdot)$ and $O_d(\cdot)$, the constants $C$ and $p_0$ may depend only on $\mathcal G$ and $d$, respectively.
If an interval for $p$ is explicitly specified, the corresponding $O$-estimate is understood to hold uniformly in the interval.
For convenience, we retain the full polynomial expressions arising from differentiation.
Their terms of degree at least $2r+1$ can equivalently be absorbed into the remainder.

The required analytic tools are the abstract polymer models and cluster expansion.

Cluster expansions are classical tools for studying thermodynamic quantities through series of connected contributions.
Gruber and Kunz~\cite{GK1971} developed a polymer system in which an activity is assigned to each polymer and an incompatibility relation is defined on the polymer set.
Depending on the model, polymers may represent connected particle groups, contours, or defects relative to a reference configuration~\cite{FP2007}.
The abstract system retains only the polymer set, the incompatibility relation, and the activities, without prescribing the internal structure of a polymer.
Koteck\'y and Preiss~\cite{KP1986} established a general sufficient condition for convergence, while Fern\'andez and Procacci~\cite{FP2007} developed sharper convergence criteria using tree-graph expansions.
Applications include expansions and estimates for partition functions and correlation functions in classical particle systems~\cite{JT2020,PT2012,PT2015}, as well as extensions to quantum gases~\cite{PU2009}.

The cluster expansion is a powerful and classical tool in the rigorous study of statistical mechanics.
It has also been applied in graph theory.
Jenssen, Perkins and Potukuchi~\cite{JP2020, JPP2022} used polymer models and cluster expansions in the enumeration and structural analysis of independent sets in the hypercube.
Davies, Jenssen and Perkins~\cite{DJP2021} applied this method to give a partial proof of the ``Upper Matching Conjecture'' of Friedland \textit{et al.}  \cite{FKM2008} and to maximize and minimize the number of independent sets and matchings of a given size in regular graphs of a given  girth.

Let $\mathcal{G}$ be a finite $d$-regular, $r$-uniform hypergraph.  We always assume that $p=\lambda/(1+\lambda)$,
that is, $\lambda=p/(1-p)$.
For brevity, let $
U_F = \bigcup_{e\in F} v(e)$
where $F\subseteq E(\mathcal{G})$. Define an auxiliary polynomial
\begin{equation}\label{eq:Q-definition}
Q_{\mathcal{G}}(p)
=\sum_{F\subseteq E(\mathcal{G})}
(-1)^{|F|}p^{|U_F|},
\end{equation}
where the empty hyperedge set contributes $1$.
Moreover, let
\[
g_{\mathcal{G}}(p)=\frac{\log Q_{\mathcal{G}}(p)}{|V(\mathcal{G})|}.
\]

Inspired by the cluster-expansion approach of Davies, Jenssen, and Perkins~\cite{DJP2021, JP2020}, we use this tool to establish small fugacity upper bounds on the occupancy and variance fractions, respectively, of linear regular uniform hypergraphs by comparing low-order coefficients and controlling the remainder uniformly in the number of vertices.
In Section~\ref{sec:cluster} we develop the polymer representation.
And in Section~\ref{sec:main-proofs} we establish coefficient bounds and prove these two results.
In Section~\ref{sec:comparison} two relevant results are provided to this topic.


\section{The abstract polymer system and the cluster expansion} \label{sec:cluster}

\ \indent In this section we establish some preliminary results related to the function $Q_{\mathcal G}(p)$ and introduce the cluster expansion and abstract polymer models.
Since the property of linearity is not yet used,  we denote by $\mathcal{G}$  a $d$-regular $r$-uniform hypergraph throughout this section.

\begin{lemma}\label{lem:normalization}
For every finite hypergraph $\mathcal{G}$, we have
\begin{equation}\label{eq:normalization}
Z_{\mathcal{G}}(\lambda)
=(1+\lambda)^{|V(\mathcal{G})|}Q_{\mathcal{G}}(p),
\end{equation}
\begin{equation}\label{eq:F-normalized}
F_{\mathcal{G}}(\lambda)=\log(1+\lambda)+g_{\mathcal{G}}(p).
\end{equation}
\end{lemma}

\begin{proof}
Write $n=|V(\mathcal{G})|$.
For each $F\subseteq E(\mathcal{G})$,
combining $p=\lambda/(1+\lambda)$ and \eqref{eq:Q-definition}, we have
\begin{align*}
(1+\lambda)^n Q_{\mathcal{G}}(p)
&=(1+\lambda)^n
	\sum_{F\subseteq E(\mathcal{G})}
	(-1)^{|F|}p^{|U_F|}\\
&=\sum_{F\subseteq E(\mathcal{G})}
	(-1)^{|F|}(1+\lambda)^n
	\left(\frac{\lambda}{1+\lambda}\right)^{|U_F|}\\
&=\sum_{F\subseteq E(\mathcal{G})}
	(-1)^{|F|}\lambda^{|U_F|}
	(1+\lambda)^{n-|U_F|}.
\end{align*}
The set $V(\mathcal{G})\setminus U_F$ has $n-|U_F|$ vertices.
Then
$ (1+\lambda)^{n-|U_F|} =\sum_{J\subseteq V(\mathcal{G})\setminus U_F}\lambda^{|J|}.$
Consequently,
\begin{align}\label{eq:normalization-containing-S}
\lambda^{|U_F|}(1+\lambda)^{n-|U_F|} =\sum_{J\subseteq V(\mathcal{G})\setminus U_F} \lambda^{|U_F|+|J|}
    =\sum_{U_F\subseteq S\subseteq V(\mathcal{G})} \lambda^{|S|}.
\end{align}
Hereafter, $S$ denotes an arbitrary (not necessarily independent) vertex subset of $\mathcal{G}$.

Substituting \eqref{eq:normalization-containing-S} into the above expression of $(1+\lambda)^n Q_{\mathcal{G}}(p)$ gives
\begin{align*}
(1+\lambda)^n Q_{\mathcal{G}}(p)
=\sum_{F\subseteq E(\mathcal{G})}(-1)^{|F|}
  \sum_{U_F\subseteq S\subseteq V(\mathcal{G})}
  \lambda^{|S|}
\end{align*}
where both sums are finite.
Therefore we have
\begin{equation}\label{eq:normalization-swap}
(1+\lambda)^n Q_{\mathcal{G}}(p)
=\sum_{S\subseteq V(\mathcal{G})}\lambda^{|S|}
	\sum_{\substack{F\subseteq E(\mathcal{G})\\U_F\subseteq S}}
	(-1)^{|F|}.
\end{equation}
For a fixed hyperedge set $F\subseteq E(\mathcal{G})$, we now have the equivalences
\begin{align*}
U_F\subseteq S
&\Longleftrightarrow \bigcup_{e\in F}v(e)\subseteq S
\Longleftrightarrow v(e)\subseteq S\text{ for every }e\in F
\Longleftrightarrow F\subseteq E(\mathcal{G}[S]).
\end{align*}
Thus \eqref{eq:normalization-swap} becomes
\begin{equation}\label{eq:normalization-induced-sum}
(1+\lambda)^n Q_{\mathcal{G}}(p)
=\sum_{S\subseteq V(\mathcal{G})}\lambda^{|S|}
	\sum_{F\subseteq E(\mathcal{G}[S])}(-1)^{|F|}.
\end{equation}
For brevity, let $m_S=|E(\mathcal{G}[S])|$.
If $m_S=0$, the only subset of $E(\mathcal{G}[S])$
is $F=\varnothing$, whose contribution is $1$.
If $m_S\ge1$, grouping subsets $F$ according to their cardinality gives
$$
\sum_{F\subseteq E(\mathcal{G}[S])}(-1)^{|F|} =\sum_{j=0}^{m_S}\binom{m_S}{j}(-1)^j
=(1-1)^{m_S}=0.
$$
Therefore,
$$
\sum_{F\subseteq E(\mathcal{G}[S])}(-1)^{|F|}
=
\begin{cases}
1,&m_S=0,\\
0,&m_S\ge1.
\end{cases}
$$
The condition $m_S=0$ means that no hyperedge of $\mathcal{G}$ is contained in $S$, that is, $S\in\cI(\mathcal{G})$.
Thus \eqref{eq:normalization-induced-sum} becomes
\begin{align*}
(1+\lambda)^n Q_{\mathcal{G}}(p)
=\sum_{S\subseteq V(\mathcal{G})}\lambda^{|S|} 	
\boldsymbol 1_{\{S\in\cI(\mathcal{G})\}}
=\sum_{I\in\cI(\mathcal{G})}\lambda^{|I|} =Z_{\mathcal{G}}(\lambda),
\end{align*}
which proves \eqref{eq:normalization}.

Finally, since the Equation \eqref{eq:normalization} is positive, $(1+\lambda)^{|V(\mathcal G)|}$ and $Q_{\mathcal G}(p)$ are positive.
Therefore,
$ \log Z_{\mathcal{G}}(\lambda) =n\log(1+\lambda)+\log Q_{\mathcal{G}}(p). $
So we have
\begin{align*}
F_{\mathcal{G}}(\lambda)
=\frac1n \log Z_{\mathcal{G}}(\lambda) =\log(1+\lambda)+\frac1n \log Q_{\mathcal{G}}(p)
=\log(1+\lambda)+g_{\mathcal{G}}(p),
\end{align*}
which is \eqref{eq:F-normalized}.
\end{proof}

Let $L(\mathcal{G})$ be the line graph of hypergraph $\mathcal{G}$: its vertex set is $E(\mathcal{G})$, and two distinct vertices are adjacent when the corresponding hyperedges intersect.
For any $e\in E(\mathcal{G})$ and each $v\in v(e)$, there are exactly $d-1$ hyperedges other than $e$ that contain $v$.
Every neighbor of $e$ in $L(\mathcal{G})$ is a hyperedge containing at least one vertex $v\in v(e)$.
Thus
$ d_{L(\mathcal{G})}(e) \le\sum_{v\in e}(d-1) =r(d-1).$
So
\begin{equation}\label{eq:line-degree}
\Delta(L(\mathcal{G}))\le r(d-1)=A_0-1.
\end{equation}


Next we will define the abstract polymer system that will be used in this paper.
Here we use the abstract framework of Koteck\'y and Preiss \cite{KP1986}, restricted to a finite polymer set.

\begin{definition}
\label{def:abstract-polymer-model}
\cite{GK1971,KP1986}
A finite abstract polymer system consists of a finite set $\mathcal P$, a symmetric and reflexive relation $\not\sim$ on $\mathcal P$ with a function $z:\mathcal P\to\mathbb C$.
The elements $\gamma$ of $\mathcal P$ are called \emph{polymers}, and $z(\gamma)$ is called the \emph{activity} of $\gamma$.
Two polymers $\gamma,\gamma'$ are \emph{incompatible} if $\gamma\not\sim\gamma'$;
otherwise, they are \emph{compatible}, written as $\gamma\sim\gamma'$.
In particular, every polymer is incompatible with itself.
A subset $\Gamma\subseteq\mathcal P$ is called \emph{compatible} if every two distinct members of $\Gamma$ are compatible.
The partition function of this system is
$\Xi_{\mathcal P}(\boldsymbol z)
= \sum\limits_{\Gamma'\subset\mathcal P} \prod\limits_{\gamma\in\Gamma'}z(\gamma)
$
where $\Gamma'$ is compatible and the empty  family contributes $1$.
\end{definition}

Now  we define polymers on the line graph of a hypergraph $\mathcal{G}$ and study their related properties.
A \emph{polymer} is a nonempty hyperedge set $\gamma\subseteq E(\mathcal{G})$ such that the subhypergraph with vertex set $\bigcup_{e\in\gamma}v(e)$ and hyperedge set $\gamma$ is connected.
Equivalently, a nonempty hyperedge set $\gamma$ is a polymer if and only if the induced subgraph $L(\mathcal{G})[\gamma]$ of the line graph $L(\mathcal{G})$ is connected.
Let $\mathcal P(\cG)$ be the set of polymers.
Sometimes we denote $\mathcal P(\cG)$ by $\mathcal P$ if $\cG$ is clear from context.
For brevity, let
$U(\gamma)=\bigcup_{e\in\gamma}v(e)$ and
define the activity of $\gamma$ as
$w_p(\gamma):=(-1)^{|\gamma|}p^{|U(\gamma)|}$.
Two polymers $\gamma,\gamma'$ are \emph{compatible}, written as $\gamma\sim\gamma'$,
if the sets of vertices they cover in $\mathcal{G}$ are disjoint, that is, $U(\gamma)\cap U(\gamma')=\varnothing$.
Equivalently, two polymers $\gamma,\gamma'$ are compatible if and only if they are disjoint as vertex sets of $L(\mathcal{G})$ and no edge of $L(\mathcal{G})$ joins a vertex of $\gamma$ to a vertex of $\gamma'$.
In all other cases they are \textit{incompatible}, that is, $U(\gamma)\cap U(\gamma')\neq\varnothing$;
 denoted by $\gamma\not\sim\gamma'$.
In particular, incompatibility is reflexive. So a compatible family cannot contain two copies of the same polymer.
Let $\mathcal P(\mathcal G)^m$ denote the finite set of all ordered polymer tuples $\widetilde{\boldsymbol\gamma}=(\gamma_1,\ldots,\gamma_m)$ of $\mathcal G$, that is, $\widetilde{\boldsymbol\gamma}\in\mathcal P(\mathcal G)^m$ and let $\Gamma$ be a subset of $\mathcal P(\mathcal G)$.

Let $\widetilde{\boldsymbol\gamma}=(\gamma_1,\ldots,\gamma_m)$.
Define the simple graph $G_{\widetilde{\boldsymbol\gamma}}$ with vertex set $[m]=\{1,\ldots,m\}$ and edge set
\(
\bigl\{\{i,j\}:1\le i<j\le m, U(\gamma_i)\cap U(\gamma_j)\ne\varnothing\bigr\}.
\)
Thus two distinct vertices of $G_{\widetilde{\boldsymbol\gamma}}$ are adjacent precisely when
the corresponding polymers are incompatible.
The \emph{Ursell coefficient} \cite{FP2007} of $\widetilde{\boldsymbol\gamma}$ is
\begin{equation}\label{eq:ursell}
\phi(\widetilde{\boldsymbol\gamma}) = \sum_{H\subseteq G_{\widetilde{\boldsymbol\gamma}}} (-1)^{|E(H)|},
\end{equation}
where $H$ is a connected spanning subgraph of $G_{\widetilde{\boldsymbol\gamma}}$.
A graph consisting of a single vertex is viewed as connected, so $\phi(\gamma_1)=1$.
If $G_{\widetilde{\boldsymbol\gamma}}$ is disconnected, the sum is empty with $\phi(\widetilde{\boldsymbol\gamma})=0$.

Below we show that the polynomial $Q_{\mathcal G}(p)$, which is defined in $(\ref{eq:Q-definition})$, is the abstract polymer partition function.
\begin{lemma}\label{lem:connected-decomposition}
Let $\mathcal{G}$ be a $d$-regular $r$-uniform hypergraph with $Q_{\mathcal{G}}(p)$ as defined in $(\ref{eq:Q-definition})$. Then
\begin{equation}\label{eq:polymer-partition}
Q_{\mathcal{G}}(p) =\sum_{\Gamma'\subseteq\mathcal P(\mathcal G)} \prod_{\gamma\in\Gamma'}w_p(\gamma),
\end{equation}
where $w_p(\gamma)=(-1)^{|\gamma|}p^{|U(\gamma)|}$,
the sum ranges over all subsets consisting of pairwise compatible  edge subsets of $\mathcal{G}$,
 and the empty family contributes $1$.
\end{lemma}
\begin{proof}
Fix $F\subseteq E(\mathcal{G})$ and let $\gamma_1,\ldots,\gamma_t$ be the vertex sets of the connected components of the induced graph $L(\mathcal{G})[F]$.
Each $\gamma_i$ is nonempty and induces a connected subgraph of $L(\mathcal{G})$, so it is a polymer.
If $i\ne j$, then $\gamma_i\cap\gamma_j=\varnothing$.
Moreover, no edge of $L(\mathcal{G})$ joins $\gamma_i$ to $\gamma_j$, because such an edge would put them in the same connected component.
Thus the polymers $\gamma_1,\ldots,\gamma_t$ are pairwise compatible.

The vertex unions of the polymers $\gamma_1,\ldots,\gamma_t$ in the original hypergraph $\mathcal{G}$ are also disjoint.
Indeed, if a vertex $v$ belongs to both $U(\gamma_i)$ and $U(\gamma_j)$, then there are hyperedges $e\in\gamma_i$ and $f\in\gamma_j$ with $v\in e\cap f$.
The hyperedges $e$ and $f$ are then adjacent in $L(\mathcal{G})$, contradicting the above argument.
Therefore, from
$F=\mathbin{\bigcup}_{i=1}^t\gamma_i$
and $U_F=\mathbin{\bigcup}_{i=1}^tU(\gamma_i)$,
we have
$ |F|=\sum_{i=1}^t|\gamma_i| $
and
$ |U_F|=\sum_{i=1}^t|U(\gamma_i)|. $
Thus
\begin{align*}
(-1)^{|F|}p^{|U_F|}
&=(-1)^{\sum_i|\gamma_i|} p^{\sum_i|U(\gamma_i)|}
    =\prod_{i=1}^t (-1)^{|\gamma_i|}p^{|U(\gamma_i)|}
    =\prod_{i=1}^t w_p(\gamma_i).
\end{align*}

Conversely, let $\Gamma'\subseteq\mathcal P(\mathcal{G})$ be formed of pairwise compatible polymers with
$F=\bigcup_{\gamma\in\Gamma'}\gamma$.
The members of $\Gamma'$ are disjoint connected vertex sets of $L(\mathcal{G})[F]$, with no edges between distinct members.
They are therefore exactly the vertex sets of the connected components of $L(\mathcal{G})[F]$.

Thus the sum over all hyperedge subsets $F\subseteq E(\mathcal{G})$ can be written as a sum over all pairwise compatible families.
Each edge subset is counted exactly once,
since its decomposition into connected components in the line graph is unique.
The empty edge set corresponds to the empty product and contributes $1$.
Therefore, we have
\[
\sum_{F\subseteq E(\mathcal{G})} (-1)^{|F|}p^{|\bigcup_{e\in F}v(e)|}
=\sum_{\Gamma'} \prod_{\gamma\in\Gamma'}w_p(\gamma),
\]
which is \eqref{eq:polymer-partition}.
\end{proof}

By Lemma~\ref{lem:connected-decomposition} and the definition of $\Xi_{\mathcal P}(\boldsymbol z)$, $Q_{\mathcal{G}}(p)$ is the partition function of the abstract polymer model with activity $z(\gamma)=w_p(\gamma)$.
Moreover, by \eqref{eq:normalization}, we have $Q_{\cG}(p)>0$.
Although $p>0$ in the original model, $Q_{\mathcal G}(p)$ is a polynomial with $Q_{\mathcal G}(0)=1$.
We can  extend $\log Q_{\mathcal G}(p)$ continuously to $p=0$ by setting $\log Q_{\mathcal G}(0)=0$. Then we have the following results.

\begin{lemma}
\label{lem:cluster-expansion}
\cite{FP2007, KP1986}
For the polymer model defined above, let $\zeta$ be as in
\eqref{eq:A0-rho}.
For every real $p\in[0,\zeta]$, we have
\begin{equation*}\label{eq:cluster-log}
\log Q_{\mathcal G}(p) =
\sum_{m\ge1}\frac1{m!} \sum_{\widetilde{\boldsymbol\gamma}\in\mathcal P(\mathcal G)^m}
\phi(\widetilde{\boldsymbol\gamma}) \prod_{i=1}^m w_p(\gamma_i),
\end{equation*}
where the series converges absolutely.
\end{lemma}

\begin{lemma}\cite{FP2007, KP1986}
\label{lem:KP-pinned}
Fix a real $p\in[0,1)$. Suppose that, for every
$\gamma\in\mathcal P(\mathcal G)$,
\begin{equation}\label{eq:KP-condition}
\sum_{\substack{\gamma'\in\mathcal P(\mathcal G)\\ \gamma'\not\sim\gamma}}
|w_p(\gamma')|e^{|\gamma'|}\le|\gamma|.
\end{equation}
Then, for every fixed $\gamma_0\in\mathcal P(\mathcal G)$,
\begin{equation}\label{eq:standard-pinned-bound}
\sum_{n\ge0}\frac1{n!}
\sum_{(\gamma_1,\ldots,\gamma_n)\in\mathcal P(\mathcal G)^n}
|\phi(\gamma_0,\gamma_1,\ldots,\gamma_n)|
\prod_{i=1}^n|w_p(\gamma_i)|
\le e^{|\gamma_0|},
\end{equation}
where the term with $n=0$ is $1$.
\end{lemma}

\begin{lemma}\label{lem:abstract-cluster}
Fix a real number $p\in[0,1)$. If
$\sum\limits_{\substack{\gamma'\in\mathcal P(\mathcal G)\\ \gamma'\not\sim\gamma}}
|w_p(\gamma')|e^{|\gamma'|}\le|\gamma|$ for every
$\gamma\in\mathcal P(\mathcal G)$,
then
\begin{equation}\label{eq:absolute-cluster}
\sum_{m\ge1}\frac1{m!}
\sum_{\widetilde{\boldsymbol\gamma}\in\mathcal P(\mathcal G)^m}
|\phi(\widetilde{\boldsymbol\gamma})| \prod_{i=1}^m|w_p(\gamma_i)|
\le \sum_{\gamma\in\mathcal P(\mathcal G)} |w_p(\gamma)|e^{|\gamma|}.
\end{equation}
\end{lemma}
\begin{proof}
For every fixed $\gamma_0\in\mathcal P(\mathcal G)$, let $A(\gamma_0)$ denote the left-hand side of
\eqref{eq:standard-pinned-bound}, i.e.,
\[
A(\gamma_0) := \sum_{n\ge0}\frac1{n!}
\sum_{(\gamma_1,\ldots,\gamma_n)\in\mathcal P^n}
|\phi(\gamma_0,\gamma_1,\ldots,\gamma_n)| \prod_{i=1}^n|w_p(\gamma_i)|.
\]
For $m\ge 1$, define
$
\sigma_m=\frac1{m!}
\sum_{\widetilde{\boldsymbol\gamma}\in\mathcal P^m}
|\phi(\widetilde{\boldsymbol\gamma})|\prod_{i=1}^m|w_p(\gamma_i)|$.

Reindexing with $m=n+1$ and relabelling the tuple entries, we obtain
\begin{align}
\sum_{\gamma_0\in \mathcal P(\mathcal G)}|w_p(\gamma_0)|A(\gamma_0) =&\sum_{n\ge0}\frac1{n!}
\sum_{\gamma_0\in\mathcal P(\mathcal G)} \sum_{(\gamma_1,\ldots,\gamma_n)\in\mathcal P(\mathcal G)^n}
|\phi(\gamma_0,\ldots,\gamma_n)|
\prod_{i=0}^n|w_p(\gamma_i)|\notag\\
=&\sum_{n\ge0}\frac1{n!}
\sum_{(\gamma_0,\ldots,\gamma_n)\in\mathcal P(\mathcal G)^{n+1}}
|\phi(\gamma_0,\ldots,\gamma_n)|
\prod_{i=0}^n|w_p(\gamma_i)|\notag\\
=&\sum_{m\ge1}\frac1{(m-1)!}
\sum_{(\gamma_1,\ldots,\gamma_m)\in\mathcal P(\mathcal{G})^m}
|\phi(\widetilde{\boldsymbol\gamma})| \prod_{i=1}^m|w_p(\gamma_i)| \notag\\
=&\sum_{m\ge1}m\sigma_m. \label{wA=msig}
\end{align}

By Lemma \ref{lem:KP-pinned}, we obtain
\begin{align}\label{wA<we}
\sum_{\gamma_0\in\mathcal P(\mathcal G)} |w_p(\gamma_0)|A(\gamma_0)
\le \sum_{\gamma_0\in\mathcal P(\mathcal G)} |w_p(\gamma_0)|e^{|\gamma_0|}.
\end{align}

Since $\sigma_m\ge0$ and $m\ge1$, by \eqref{wA=msig} and \eqref{wA<we}, we have
\begin{align*}
\sum_{m\ge1}\sigma_m &\le \sum_{m\ge1}m\sigma_m
    = \sum_{\gamma_0\in\mathcal P(\mathcal G)} |w_p(\gamma_0)|A(\gamma_0) \\
&\le \sum_{\gamma_0\in\mathcal P(\mathcal G)} |w_p(\gamma_0)|e^{|\gamma_0|}.
\end{align*}
This proves \eqref{eq:absolute-cluster}.
\end{proof}

\section{Proofs of Main Results}
{\bl \label{sec:main-proofs}}

In this section we will prove two main results.
We first provide some lemmas that will be used.

\begin{lemma}\cite{Jukna2011}
\label{nm}
Let $a$ and $b$ be positive integers, and $e$ be the base of the natural logarithm. Then
$$
\binom{a}{b}\le\left(\frac{ea}{b}\right)^b
$$
\end{lemma}

\begin{lemma}\cite{BCKL2013}
\label{refBCKL}
Let $G$ be a graph with maximum degree $\Delta$ with $v \in V(G)$.
The number of connected induced subgraphs of $G$ with $k$ vertices containing $v$ is at most $\frac{1}{k}\binom{k\Delta}{k-1}$.
\end{lemma}

\begin{lemma}
\label{lem:connected-count}
Let $G$ be a graph of maximum degree at most $\Delta$ with $v \in V(G)$.
Let
$ c_{G,v}(k)=\left|\left\{ S\subseteq V(G): |S|=k,\ v\in S,\ G[S]\text{ is connected} \right\}\right|. $
Then
$c_{G,v}(k)
\le \bigl(4(\Delta+1)\bigr)^{k-1}$.
\end{lemma}
\begin{proof}
Here $e$ denotes the base of the natural logarithm.
By Lemma \ref{refBCKL}, we have
$ c_{G,v}(k)~\le~\frac1k\binom{k\Delta}{k-1}. $
Next, it suffices to prove that
\begin{align}\label{connetseteq1}
\frac1k\binom{k\Delta}{k-1} \le \bigl(4(\Delta+1)\bigr)^{k-1}.
\end{align}
When $k=1$, both sides equal $1$.
When $\Delta=0$ and $k\ge2$, no connected $k$-vertex subset can contain
$v$, so $c_{G,v}(k)=0$. Thus (\ref{connetseteq1}) holds trivially.

Now suppose that $k\ge2$ and $\Delta\ge1$.
Using Lemma \ref{nm} with $a=k\Delta$ and $b=k-1$, we have
\begin{align}\label{connetseteq2}
\frac1k\binom{k\Delta}{k-1}
\le \frac1k \left(\frac{ek\Delta}{k-1}\right)^{k-1}.
\end{align}

Note that $\binom{k-1}{j}\le(k-1)^j$. Then
$
\left(1+\frac1{k-1}\right)^{k-1}
=\sum_{j=0}^{k-1} \binom{k-1}{j}\frac1{(k-1)^j} \le \sum_{j=0}^{k-1}1
=k$, implying
\begin{align}\label{connetseteq4}
\frac1k\left(\frac{k}{k-1}\right)^{k-1}\le1.
\end{align}
Combining \eqref{connetseteq2} and \eqref{connetseteq4}, we have
$ \frac1k\binom{k\Delta}{k-1} \le \frac1k \left(\frac{k}{k-1}\right)^{k-1}(e\Delta)^{k-1}
\le(e\Delta)^{k-1}$.
Since $e<4$ and $\Delta\le\Delta+1$, we have
$(e\Delta)^{k-1} \le \bigl(4(\Delta+1)\bigr)^{k-1}.$
This completes the proof.
\end{proof}

Recall that
$g_{\mathcal G}(p)=|V(\mathcal G)|^{-1}\log Q_{\mathcal G}(p)$.
By Lemma~\ref{lem:cluster-expansion}, $g_{\mathcal G}$ admits an absolutely convergent power series expansion on $[0,\zeta]$, and all its terms are monomials in $p$.
Rearranging these terms as a power series in $p$, we obtain:
\begin{equation}\label{expansion_Q}
g_{\mathcal G}(p)
=\sum_{j\ge1}b_j(\mathcal G)p^j,
\qquad 0\le p\le\zeta.
\end{equation}
Here $b_j(\mathcal G)$ denotes the coefficient of $p^j$,
and the constant term vanishes because $g_{\mathcal G}(0)=0$.

\begin{proposition}
\label{prop:coefficient-bound}
Let $\mathcal G$ be a $d$-regular, $r$-uniform hypergraph.
Recall that $\zeta=(32A_0)^{-d/r}$ with $A_0=r(d-1)+1$ and $M=\frac{3d}{20rA_0}$.
Then
\begin{equation}\label{eq:weighted-coefficient-bound}
\sum_{j\ge1}|b_j(\mathcal{G})|\zeta^j \le M
\end{equation}
with
$|b_j(\mathcal{G})|\le M\zeta^{-j}$ for each $j\geq 1$.
\end{proposition}

\begin{proof}
In this proof, we denote by $e$ the base of the natural logarithm and suppose that $|p|\le\zeta$.
Denote by $\Gamma_f$ the set of the polymers containing $f\in E(\mathcal G)$.
Recall that $A_0=r(d-1)+1$ and $\zeta=(32A_0)^{-d/r}$, then $0<\zeta<1$.
Suppose that a polymer $\gamma$ contains $k$ hyperedges.
Let $d_\gamma(v)=|\{f\in\gamma:v\in f\}|$ for $v\in U(\gamma)$.
The number of the pairs $(v,f)$ with $f\in\gamma$ and $v\in f$ is $kr$.
Note that
$kr=\sum_{f\in\gamma}|f|
=\sum_{v\in U(\gamma)}d_\gamma(v).$
Since $\mathcal{G}$ is $d$-regular,  we have $d_\gamma(v)\le d$ for every $v$.
Therefore
$ kr\le d|U(\gamma)|,$
which is equivalent to
$ |U(\gamma)|\ge\frac{kr}{d}. $
Let $x=|p|^{r/d}$.
Since $|p|<1$, then
\begin{equation}\label{eq:polymer-weight-size}
|w_p(\gamma)|
=|p|^{|U(\gamma)|}
\le |p|^{kr/d}
=x^k.
\end{equation}
Fix a hyperedge $f\in E(\mathcal{G})$.
A polymer of size $k$ containing $f$ is exactly a connected $k$-vertex subset of $L(\mathcal{G})$ containing the vertex $f$.
By Lemma~\ref{lem:connected-count}, \eqref{eq:line-degree}, and \eqref{eq:polymer-weight-size}, we have
\begin{align}
\sum_{\gamma\in \Gamma_f}|w_p(\gamma)|e^{|\gamma|}
&=\sum_{k\ge 1} 	\sum_{\substack{\gamma\in \Gamma_f \\ |\gamma|=k}} |w_p(\gamma)|e^k
    \le \sum_{k\ge1} x^k e^k	\left( \sum_{\substack{\gamma\in \Gamma_f \\ |\gamma|=k}} 1 \right) \notag\\
&\le\sum_{k\ge1}(4A_0)^{k-1}e^kx^k
=ex\sum_{k\ge1}(4A_0ex)^{k-1} \notag\\
&=\frac{ex}{1-4A_0ex}.
\label{eq:rooted-polymer-sum}
\end{align}
Since $0\leq x= |p|^{\frac rd} \leq\zeta^{\frac rd}=(32A_0)^{-1}<\frac{1}{4A_0e}$ from $|p|\le \zeta$,
we have $4eA_0x\le e/8<1$ and the last equality of \eqref{eq:rooted-polymer-sum} holds.
For brevity, write $B_0:=\frac{e}{32A_0(1-e/8)}.$
Therefore, \eqref{eq:rooted-polymer-sum} becomes
\begin{align}\label{eq:rooted-uniform-bound}
\sum_{\gamma\in\Gamma_f}|w_p(\gamma)|e^{|\gamma|} \le& \frac{ex}{1-4A_0ex} \le \frac{e}{32A_0(1-e/8)} = B_0.
\end{align}

Now fix an arbitrary polymer $\gamma\in\mathcal P(\mathcal G)$.
Then \eqref{eq:rooted-uniform-bound} holds for every $f\in E(\cG)$, so we may apply it to each $f\in N_{L(\mathcal G)}[\gamma]$.
Let
\begin{align*}
N_{L(\mathcal G)}[\gamma]
=&\gamma\cup \left\{ j\in E(\mathcal G)\setminus\gamma: j\cap f\ne\varnothing \text{ for some }f\in\gamma \right\},\\
N_{L(\mathcal G)}[f]
=&\{f\} \cup \left\{ j\in E(\mathcal G)\setminus \{f\}: j\cap f\ne\varnothing \right\}.
\end{align*}
By \eqref{eq:line-degree}, we have $|N_{L{(\mathcal{G})}}[f]|\le \Delta(L(\mathcal{G}))+1\le A_0$ for $f\in V(L(\mathcal{G}))$.
Then
\begin{equation}\label{eq:closed-neighborhood-size}
|N_{L{(\mathcal{G})}}[\gamma]|
\le\sum_{f\in\gamma}|N_{L{(\mathcal{G})}}[f]|
\le A_0|\gamma|.
\end{equation}
If $\gamma'\not\sim\gamma$, then vertex subset $\gamma'$ of $L(\cG)$ contains some vertices $f\in N_{L(\cG)}[\gamma]$. \eqref{eq:rooted-uniform-bound} and \eqref{eq:closed-neighborhood-size} give
\begin{align*}
\sum_{\gamma'\not\sim\gamma} |w_p(\gamma')|e^{|\gamma'|}
&\le\sum_{f\in N_{L(\cG)}[\gamma]}~ \sum_{\gamma'\in \Gamma_f}|w_p(\gamma')|e^{|\gamma'|}\\
&\le |N_{L(\mathcal{G})}[\gamma]|B_0 \le A_0|\gamma|B_0\\
&=|\gamma|\frac{e}{32(1-e/8)}.
\end{align*}

Define a function $g(x)=\frac{x}{1-x/8}$ with $0<x<8$. Then $g(x)$ is monotonically increasing.
Therefore, by $e<3$, we have
\begin{equation}\label{eq:numerical-cluster-constant}
\frac{e}{32(1-e/8)} <\frac{3}{32(1-3/8)} =\frac3{20}<1.
\end{equation}
Thus, we obtain $
\sum_{\gamma'\not\sim\gamma} |w_p(\gamma')|e^{|\gamma'|}
\le |\gamma|$, meaning that the condition \eqref{eq:KP-condition} holds if $|p|\le\zeta$.

Since every polymer is nonempty, $|\gamma|\ge 1$.
At $|p|=\zeta$, by \eqref{eq:rooted-uniform-bound} and Lemma~\ref{lem:abstract-cluster},
we have
\begin{align}
\sum_{m\ge1}\frac1{m!} \sum_{\widetilde{\boldsymbol\gamma}\in \mathcal P^m} |\phi(\widetilde{\boldsymbol\gamma})|  \prod_{i=1}^m\zeta^{|U(\gamma_i)|}
&\le \sum_{\gamma}\zeta^{|U(\gamma)|}e^{|\gamma|}
  \le \sum_{\gamma}|\gamma| \zeta^{|U(\gamma)|}e^{|\gamma|} \notag\\
&= \sum_{f\in E(\mathcal{G})}  \sum_{\gamma\in\Gamma_f}\zeta^{|U(\gamma)|}e^{|\gamma|}
  \le \sum_{f\in E(\mathcal{G})}  B_0 \notag\\
&\le |E(\mathcal{G})|B_0. \label{eq:all-polymer-root-bound}
\end{align}

Finally, we consider the number of all incident vertex-edge pairs of $\mathcal G$, which is $r|E(\mathcal{G})|$ when counted by edges and $d|V(\mathcal{G})|$ when counted by vertices.
Thus $r|E(\mathcal{G})|=d|V(\mathcal{G})|,$ i.e.,
$ |E(\mathcal{G})|\big/|V(\mathcal{G})|=d/r. $
Dividing \eqref{eq:all-polymer-root-bound} by $|V(\mathcal{G})|$ and using \eqref{eq:numerical-cluster-constant},
we have
\begin{align}\label{eq:absolute-cluster-per-vertex}
\frac{1}{|V(\mathcal{G})|}\sum_{m\ge1}\frac1{m!} \sum_{\widetilde{\boldsymbol\gamma}\in \mathcal P^m} |\phi(\widetilde{\boldsymbol\gamma})| \prod_{i=1}^m\zeta^{|U(\gamma_i)|}
&\le \frac{|E(\mathcal{G})|}{|V(\mathcal{G})|}B_0
  =\frac dr\frac{e}{32A_0(1-e/8)}\notag\\
&<\frac dr\frac{3}{20A_0} = M.
\end{align}

For a polymer tuple $\widetilde{\boldsymbol\gamma}=(\gamma_1,\ldots,\gamma_m)\in \mathcal P^m$,
let $q(\widetilde{\boldsymbol\gamma})=\sum\limits_{i=1}^m|U(\gamma_i)|$.
By the definition of $w_p$, we have
$\prod\limits_{i=1}^m w_p(\gamma_i)
= \prod\limits_{i=1}^m (-1)^{|\gamma_i|}p^{|U(\gamma_i)|}
= (-1)^{\sum\limits_{i=1}^m|\gamma_i|} p^{\sum\limits_{i=1}^m|U(\gamma_i)|}$.
By Lemma \ref{lem:cluster-expansion} and the definition of $g_{\mathcal G}(p)$, we have
$
g_{\mathcal G}(p) = \frac1{|V(\mathcal G)|}
  \sum\limits_{m\ge1}\frac1{m!} \sum\limits_{\widetilde{\boldsymbol\gamma}\in\mathcal P(\mathcal G)^m}
  \phi(\widetilde{\boldsymbol\gamma}) (-1)^{\sum\limits_{i=1}^m|\gamma_i|} p^{q(\widetilde{\boldsymbol\gamma})}$.

Only the tuples satisfying \(q(\widetilde{\boldsymbol\gamma})=j\) contribute to the coefficient of $p^j$.
So we obtain
\begin{align*}
b_j(\mathcal G) = \frac1{|V(\mathcal G)|}
  \sum_{m\ge1}\frac1{m!} \sum_{\substack{
  \widetilde{\boldsymbol\gamma}\in\mathcal P(\mathcal G)^m\\
  q(\widetilde{\boldsymbol\gamma})=j }}
   (-1)^{\sum_{i=1}^m|\gamma_i|} \phi(\widetilde{\boldsymbol\gamma}).
\end{align*}
For fixed $j$, each polymer contains at least one edge, so
$j=q(\widetilde{\boldsymbol\gamma})=\sum\limits_{i=1}^m|U(\gamma_i)|\ge rm$,
and hence $m\le\lfloor j/r\rfloor$.
Therefore the outer sum contains only finitely many terms.
So we have
\begin{align}\label{eq:bj}
|b_j(\mathcal G)|
&= \Bigg|\frac1{|V(\mathcal G)|}
  \sum_{m\ge1}\frac1{m!} \sum_{\substack{\widetilde{\boldsymbol\gamma}\in\mathcal P(\mathcal G)^m\\q(\widetilde{\boldsymbol\gamma})=j }}
   (-1)^{\sum_{i=1}^m|\gamma_i|} \phi(\widetilde{\boldsymbol\gamma})\Bigg| \notag\\
&\le \frac1{|V(\mathcal G)|}
  \sum_{m\ge1}\frac1{m!}
  \sum_{\substack{\widetilde{\boldsymbol\gamma}\in\mathcal P(\mathcal G)^m\\ q(\widetilde{\boldsymbol\gamma})=j }}
|\phi(\widetilde{\boldsymbol\gamma})|.
\end{align}

By $q(\widetilde{\boldsymbol\gamma})=j$, we have
$\zeta^j =\zeta^{q(\widetilde{\boldsymbol\gamma})} =\zeta^{\sum_{i=1}^m|U(\gamma_i)|} =\prod_{i=1}^m\zeta^{|U(\gamma_i)|}$.
Therefore, by \eqref{eq:bj}, we have
\begin{align}
|b_j(\mathcal G)|\zeta^j
\le \frac1{|V(\mathcal G)|}
  \sum_{m\ge1}\frac1{m!}
  \sum_{\substack{\widetilde{\boldsymbol\gamma}\in\mathcal P(\mathcal G)^m\\q(\widetilde{\boldsymbol\gamma})=j}}
  |\phi(\widetilde{\boldsymbol\gamma})| \prod_{i=1}^m\zeta^{|U(\gamma_i)|}.
\end{align}

Summing this inequality over $j$ and using \eqref{eq:absolute-cluster-per-vertex}, we have
\begin{align*}
\sum_{j\ge1}|b_j(\mathcal{G})|\zeta^j
&\le \sum_{j\ge 1}\frac1{|V(\mathcal G)|}
\sum_{m\ge1}\frac1{m!}
\sum_{\substack{
\widetilde{\boldsymbol\gamma}\in\mathcal P(\mathcal G)^m\\
q(\widetilde{\boldsymbol\gamma})=j
}}
|\phi(\widetilde{\boldsymbol\gamma})|
\prod_{i=1}^m\zeta^{|U(\gamma_i)|} \notag\\
&=
\frac1{|V(\mathcal G)|}
\sum_{m\ge1}\frac1{m!}
\sum_{ \widetilde{\boldsymbol\gamma}\in\mathcal P(\mathcal G)^m}
|\phi(\widetilde{\boldsymbol\gamma})| \prod_{i=1}^m\zeta^{|U(\gamma_i)|} \le M,
\end{align*}
which is \eqref{eq:weighted-coefficient-bound}.
In particular, the Taylor series converges absolutely for every $0<p\le\zeta$.
Since every summand on the left is nonnegative, we have
$
|b_j(\mathcal{G})|\zeta^j \le\sum_{\ell\ge1}|b_\ell(\mathcal{G})|\zeta^\ell \le M$.
Dividing by $\zeta^j>0$ gives $|b_j(\mathcal{G})|\le M\zeta^{-j}.$
\end{proof}

Let $\mathcal G$ be a finite linear, $d$-regular, $r$-uniform hypergraph.
When $r=3$, let $\tau_3$ be the number of copies of the Berge triangle in $\mathcal G$,
and $\tau_4$ be the number of copies of the Pasch configuration in $\mathcal G$.
Define
$$
\delta_{\mathcal G} :=
\begin{cases}
\dfrac{\tau_3-\tau_4}{n}, & r=3,\\[6pt]
0, & r\ge4.
\end{cases}
$$

\begin{lemma}
\label{lem:linear-expansion}
Let $\mathcal G$ be a finite linear, $d$-regular, $r$-uniform hypergraph with $r\ge3$. Then
\begin{equation}\label{eq:G1-refined-expansion}
g_{\mathcal G}(p)
=-\frac drp^r+\binom d2p^{2r-1}
-\left(\frac{dA_0}{2r}+\delta_{\mathcal G}\right)p^{2r}
+R_{\mathcal G}(p),
\end{equation} for $0<p\le\zeta$,
where
\begin{equation}\label{eq:G1-refined-remainder}
R_{\mathcal G}(p)=\sum_{k\ge2r+1}b_k(\mathcal G)p^k,
\qquad
\sum_{k\ge2r+1}|b_k(\mathcal G)|\zeta^k\le M.
\end{equation}
\end{lemma}

\begin{proof}
Let $n=|V(\mathcal G)|$ and $m=|E(\mathcal G)|$.
Considering the number of all incident vertex-edge pairs of $\mathcal G$, we have
$mr=nd$, that is, $m=\frac{nd}{r}$.

Recall that
$Q_{\mathcal{G}}(p)  =\sum_{F\subseteq E(\mathcal{G})} (-1)^{|F|}p^{|U_F|}$.
We now sort the terms of $Q_{\mathcal{G}}(p)$ according to the value of $|F|$.
The set $F=\varnothing$ contributes $1$.
If $F=\{e\}$, then $|\bigcup_{f\in F}v(f)|=|v(e)|=r$. So each of the $m$ one-edge sets contributes $-p^r$.
Therefore, their total contribution is
$-mp^r=-\frac{nd}{r}p^r.$

Next, suppose that $|F|\ge 2$ and choose distinct hyperedges $f_1,f_2\in F$.
Since $\mathcal G$ is linear, $|v(f_1)\cap v(f_2)|\le 1$.
Thus
$
\left|\bigcup_{g\in F} v(g) \right|
\ge |v(f_1)\cup v(f_2)|
=|v(f_1)|+|v(f_2)|-|v(f_1)\cap v(f_2)|
\ge 2r-1$.
Moreover,  the number of unordered intersecting edge pairs is $n\binom d2$.
Indeed, each such pair has a unique common vertex, and each vertex lies in exactly $d$ edges.
Their unions have size $2r-1$; the remaining $\binom m2-n\binom d2$ edge pairs have unions of size $2r$.
The number of vertices occupied by two edges is at most $2r$.
So, when $|F|=2$, the contribution to the coefficient of $p^{2r-1}$ is $n\binom{d}{2}$ and the contribution to the coefficient of $p^{2r}$ is $\binom{m}{2}-n\binom{d}{2}$.

For any $F$ with $|F|\ge3$, choose three distinct edges
$e_1,e_2,e_3\in F$. By the linearity and $r$-uniformity with $r\ge 3$ we have
$|U_F|\ge |e_1\cup e_2\cup e_3| \ge 3r-|e_1\cap e_2|-|e_1\cap e_3|-|e_2\cap e_3| \ge 3r-3\ge 2r$.
Thus we only need to consider the terms of degree at least $2r$. We determine which such sets can satisfy $|U_F|\le2r$
and distinguish the following two cases.

\medskip
\noindent\textbf{Case 1. $r\ge4$.}

In this case, $3r-3\ge2r+1$.
Hence every edge set $F$ with $|F|\ge3$ contributes only at the terms of degree at least $2r+1$.
Thus the coefficients through degree $2r$ are determined entirely by the edge sets with $|F|\le2$.

\medskip
\noindent\textbf{Case 2. $r=3$.}

Note that $3r-3=2r$.
So the preceding bound gives $|U_F|\ge 2r=6$.
Consequently, among the edge sets with $|F|\ge3$, only those covering exactly six vertices can affect the coefficient of $p^6$.
We consider three subcases.

\smallskip
\noindent\textbf{Subcase 2.1. $|F|=3$.}

In this subcase, $|U_F|=6$ occurs precisely when the three hyperedges are pairwise intersecting with distinct intersection vertices (such an unordered triple is a Berge triangle).
So $\tau_{3}$ is the number of these triples.
Since each has sign $(-1)^3=-1$ in $Q_{\mathcal G}(p)$, their total contribution is
$ -\tau_{3} p^6. $
Every other three-edge set covers at least seven vertices and
contributes only at the terms of at degree at least seven.

\smallskip
\noindent\textbf{Subcase 2.2. $|F|=4$.}

In this subcase, $|U_F|=6$ occurs only in one situation: the four hyperedges are pairwise intersecting  with distinct intersection vertices,
and every vertex has degree 2 in the subhypergraph formed by $F$  (these are Pasch configurations).
So $\tau_{4}$ is the number of these sets.
Each has sign $(-1)^4=1$, so their total contribution is $ \tau_{4} p^6. $
Every other four-edge set covers at least seven vertices and contributes only at the terms of at degree at least seven.

\smallskip
\noindent\textbf{Subcase 2.3. $|F|\ge5$.}

Suppose that $|F|\ge5$ and $F$ covers exactly six vertices.
Each hyperedge of $F$ containing a fixed vertex $v$ contains two other vertices.
By the linearity, two different hyperedges containing $v$ cannot share any other vertex.
Since there are only five other vertices, at most two hyperedges of $F$ can contain $v$.
Therefore, there does not exist a set of 5 hyperedges on 6 vertices.
And this subcase does not contribute to the coefficient of the term with power $p^{2r}$.

\medskip
It follows that, when $r=3$, the total contribution at the terms at degree six
from all hyperedge sets with $|F|\ge3$ is $(\tau_{4}-\tau_{3})p^6$.
We next show that $\tau_{3}-\tau_{4}$ is nonnegative.
Each Pasch configuration contains exactly four Berge triangles.
Conversely, a Berge triangle can belong to at most one Pasch configuration.
Its union already consists of six vertices, three of which have degree two within the triangle.
In any Pasch configuration containing this Berge triangle, the fourth edge must therefore consist of the other three vertices.
This determines the fourth edge uniquely.
Then
$
\tau_{3}\ge4\tau_{4}.
$
Consequently, we have, for $|F|\ge 3$,
\[
\delta_{\mathcal G} =
\begin{cases}
    \frac{\tau_{3}-\tau_{4}}{n}\ge0 & r=3,\\
    0,  & r\ge 4.
\end{cases}
\]

The preceding counts determine all coefficients of $Q_{\mathcal G}(p)$ with degree at most $2r$.
Since $Q_{\mathcal G}(p)$ is a polynomial, its remaining terms have degree at least $2r+1$.
We therefore obtain

\begin{align}\label{eq:QG1}
Q_{\mathcal G}(p) =1-mp^r+n\binom d2p^{2r-1} +\left(\binom m2-n\binom d2-n\delta_{\mathcal G}\right)p^{2r}
+O_{\mathcal G}(p^{2r+1}).
\end{align}

Since the empty set contributes
$1$ to the partition function, by \eqref{eq:normalization}, we have
$Q_{\mathcal G}(p) =(1-p)^n Z_{\mathcal G}\left(\frac{p}{1-p}\right)
\ge (1-\zeta)^n>0,$ and $0\le p\le\zeta.$
Therefore, we have
\begin{align}\label{eq:logQG1}
\log Q_{\mathcal G}(p) &= \log \left( 1+ (Q_{\mathcal G}(p)-1)\right) \notag\\
&=Q_{\mathcal G}(p)-1 -\frac12\bigl(Q_{\mathcal G}(p)-1\bigr)^2 +O_{\mathcal G}\!\left( |Q_{\mathcal G}(p)-1|^3 \right).
\end{align}

By \eqref{eq:QG1} and $r\ge 3$, we have $
\bigl(Q_{\mathcal G}(p)-1\bigr)^2 =m^2p^{2r}+O_{\mathcal G}(p^{3r-1})$,
and the third-order Taylor remainder is
$O_{\mathcal G}(p^{3r})$.
Since $3r-1\ge2r+1$ for $r\ge3$, by \eqref{eq:QG1} and \eqref{eq:logQG1}, we obtain
\begin{align*}\label{eq:logQG1-II}
\log Q_{\mathcal G}(p) &=  Q_{\mathcal G}(p)-1 -\frac{m^2}{2}p^{2r} +O_{\mathcal G}(p^{2r+1}) \notag\\
&= -mp^r+ n\binom d2p^{2r-1}
	+\left(\binom m2-n\binom d2-\frac{m^2}{2}-n\delta_{\mathcal G}\right)p^{2r}
	+O_{\mathcal G}(p^{2r+1}).
\end{align*}

\noindent
Since $\frac mn=\frac dr$, the coefficient of $p^{2r}$ in $g_{\mathcal G}(p)$ is
$$
\frac1n \left(\binom m2-n\binom d2-\frac{m^2}{2} -n\delta_{\mathcal G} \right)
=-\frac{d}{2r}-\binom d2 -\delta_{\mathcal G}
=-\frac{dA_0}{2r}-\delta_{\mathcal G}.
$$

So
$g_{\mathcal G}(p)
=-\frac drp^r+\binom d2p^{2r-1} -\left(\frac{dA_0}{2r}+\delta_{\mathcal G}\right)p^{2r} +O_{\mathcal G}(p^{2r+1})$,
which gives \eqref{eq:G1-refined-expansion}.
By Proposition~\ref{prop:coefficient-bound}, we obtain \eqref{eq:G1-refined-remainder} directly.
\end{proof}

Recall that $p=\frac{\lambda}{1+\lambda}$.
In the following we use the operator:
\begin{equation}\label{eq:refined-Dp-operator}
\mathcal D:=\lambda\frac{\mathrm d}{\mathrm d\lambda}
=p(1-p)\frac{\mathrm d}{\mathrm dp}.
\end{equation}
Direct calculation gives
\begin{equation}\label{eq:refined-monomial-derivatives}
\mathcal Dp^k=k(1-p)p^k,
\qquad
\mathcal D^2p^k=k(1-p)\bigl(k-(k+1)p\bigr)p^k.
\end{equation}

\begin{lemma}\label{lem:weighted-derivative-tail}
Let $m\ge3$ be an integer and $\zeta$ be as defined in (\ref{eq:A0-rho}).
If $ R(p)=\sum_{k\ge m}a_kp^k$ and $\sum_{k\ge m}|a_k|\zeta^k\le\omega$ for some $\omega\ge0$, then, for $0\le p\le\zeta/2$,
\begin{align}
|\mathcal DR(p)|
&\le m\omega(1-p)\left(\frac p\zeta\right)^m,
\label{eq:DR}\\
|\mathcal D^2R(p)|
&\le m^2\omega(1-p)^2\left(\frac p\zeta\right)^m.
\label{eq:D2R}
\end{align}
\end{lemma}

\begin{proof}
Since $\sum_{k\ge m}|a_k|\zeta^k\le\omega$, termwise differentiation can be performed twice.
We first prove the following claim.
\begin{claim}\label{kxk}
$kt^k\le mt^m$ and $k^2t^k\le m^2t^m$, for $k\ge m\ge 3$ and $0\le t\le 1/2$.
\end{claim}

\noindent\textit{Proof of Claim \ref{kxk}.}
If $t=0$, both inequalities hold trivially.
If $0<t\le1/2$, then
\begin{align*}
\frac{(k+1)t^{k+1}}{kt^k} =& \left(1+\frac1k\right)t \le \left(1+\frac13\right)\frac12 <1, \\
\frac{(k+1)^2t^{k+1}}{k^2t^k} =& \left(1+\frac1k\right)^2t \le \left(1+\frac13\right)^2\frac12 <1.
\end{align*}
This completes the proof of Claim \ref{kxk}.

Let $t=p/\zeta$. Then $t\le 1/2$, since $p\le\zeta/2$.
Therefore, by \eqref{eq:refined-monomial-derivatives} and Claim \ref{kxk}, we have
\begin{align*}
|\mathcal DR(p)| &\le (1-p)\sum_{k\ge m}|a_k|kp^k = (1-p)\sum_{k\ge m}|a_k|k \zeta^k \frac{p^k}{\zeta^k}\\
&\le (1-p)\sum_{k\ge m}|a_k|\zeta^k kt^k \le (1-p) mt^m \sum_{k\ge m}|a_k|\zeta^k \\
&\le (1-p)mt^m \omega.
\end{align*}
Similarly, we have
$
|\mathcal D^2R(p)| \le(1-p)^2\sum_{k\ge m}|a_k|\zeta^k k^2t^k \le m^2\omega(1-p)^2t^m$.

The assertions at $p=0$ follow directly.
This completes the proof.
\end{proof}

Before proving the theorem, we  define
\begin{align}
\theta_V(p)
:=&p(1-p)+d(1-p)\Bigl(-\bigl(r-(r+1)p\bigr)p^r +\frac{(2r-1)(d-1)}{2}\bigl(2r-1-2rp\bigr)p^{2r-1}\notag\\
&\quad-A_0\bigl(2r-(2r+1)p\bigr)p^{2r} \Bigr).\label{eq:theta-V}\\
\theta_E(p):=& p+d(1-p)\left(-p^r+\frac{(2r-1)(d-1)}{2}p^{2r-1}-A_0p^{2r}\right). \label{eq:theta-E}
\end{align}

\begin{proof}[Proof of Theorems~\ref{thm:variance-main} and \ref{thm:occupancy-main}]
The condition $0<\lambda\le\lambda_0(r,d)$ is equivalent to
$0<p\le\zeta/2$. By \eqref{eq:F-normalized} and
\eqref{eq:refined-Dp-operator}, we have
$
E_{\mathcal G}(\lambda)=p+\mathcal Dg_{\mathcal G}(p)$ and $
V_{\mathcal G}(\lambda)=p(1-p)+\mathcal D^2g_{\mathcal G}(p)
$.
Applying \eqref{eq:refined-monomial-derivatives} to the exact expression in Lemma~\ref{lem:linear-expansion}, from \eqref{eq:theta-E} and \eqref{eq:theta-V}, we have
\begin{align*}
E_{\mathcal G}(\lambda)
&= p-d(1-p)p^r+ \binom{d}{2}(2r-1)(1-p)p^{2r-1}
	-\left( \frac{dA_0}{2r}+\delta_\mathcal G  \right) 2r(1-p)p^{2r} +\mathcal D R_{\mathcal G}(p) \\
&=\theta_E(p)-2r\delta_{\mathcal G}(1-p)p^{2r} +\mathcal DR_{\mathcal G}(p), \\
V_{\mathcal G}(\lambda)
&= p(1-p) - d(1-p)\big( r-(r+1)p\big)p^r +\binom{d}{2}(2r-1)(1-p)(2r-1-2rp)p^{2r-1} \\
	&\qquad-\left(\frac{dA_0}{2r} + \delta_{\mathcal G} \right)2r(1-p)\big(2r-(2r+1)p \big) p^{2r} +\mathcal D^2R_{\mathcal G}(p) \\
&=\theta_V(p)-2r\delta_{\mathcal G}(1-p) \bigl(2r-(2r+1)p\bigr)p^{2r}
+\mathcal D^2R_{\mathcal G}(p).
\end{align*}
Both terms involving $\delta_{\mathcal G}$ are nonpositive on this interval.
Lemma~\ref{lem:weighted-derivative-tail} with $m=2r+1$ and $\omega=M$ gives
\begin{align}
E_{\mathcal G}(\lambda) &\le \theta_E(p) + \frac{M(1-p)(2r+1)}{\zeta^{2r+1}}p^{2r+1} = \theta_E(p) + O_{r,d}(p^{2r+1}). \label{EG}\\
V_{\mathcal G}(\lambda) &\le \theta_V(p) + \frac{M(2r+1)^2(1-p)^2}{\zeta^{2r+1}} p^{2r+1} = \theta_V(p) + O_{r,d}(p^{2r+1}). \label{VG}
\end{align}
Hence we prove \eqref{eq:occupancy-refined-bound} and \eqref{eq:variance-refined-bound}.
\end{proof}

\section{Corollary of Two Explicit Upper Bounds}\label{sec:comparison}

In this section, let $r\ge3$ and $d\ge2$.
Recall that $\mathcal H_d^r$ consists of $d$ marked vertices and $d$ disjoint
$(r-1)$-vertex blocks, with one hyperedge for each marked-vertex--block pair.
We compare the explicit versions of the upper bounds in
Theorems~\ref{thm:variance-main} and \ref{thm:occupancy-main} with the corresponding
fractions of $\mathcal H_d^r$.
Before presenting the corollary, we first give the definitions of the relevant parameters:
\begin{align}
\widehat p_{\mathrm{var}}(r,d)
&=\frac{5(d-1)A_0(r+1)^2}{432(r+2)^2}\zeta^{r+2}, \label{eq:comparison-pvar-hat}\\
p_{\mathrm{var}}(r,d)
&=\min\left\{\frac18,\frac\zeta2,
\widehat p_{\mathrm{var}}(r,d)\right\},
\qquad
\lambda_{\mathrm{var}}(r,d)
=\frac{p_{\mathrm{var}}(r,d)}{1-p_{\mathrm{var}}(r,d)}, \notag \\
\widehat p_{\mathrm{occ}}(r,d)
&=\frac{5(d-1)A_0(r+1)}{6(r+3)}\zeta^{r+2}, \label{eq:comparison-pocc-hat}\\
p_{\mathrm{occ}}(r,d)
&=\min\left\{\frac\zeta2,
\widehat p_{\mathrm{occ}}(r,d)\right\},
\qquad
\lambda_{\mathrm{occ}}(r,d)
=\frac{p_{\mathrm{occ}}(r,d)}{1-p_{\mathrm{occ}}(r,d)}. \notag 
\end{align}
Note that $A_0$ and $\zeta$ are as defined in \eqref{eq:A0-rho} and $p=\lambda/(1+\lambda)$.
For $r\ge3$ and $d\ge2$, $\mathcal H_d^r$ is not linear, so our following result does not identify an extremizer within the linear class.

\begin{corollary}\label{cor:variance-comparison}
Let $r\ge3$, $d\ge2$.
If $\mathcal{G}$ is a finite linear, $d$-regular, $r$-uniform hypergraph, then\begin{itemize}
\item [$(i)$] \ \  $V_{\mathcal{G}}(\lambda)\le V_{\cH_d^r}(\lambda)$ for every $0<\lambda\le\lambda_{\mathrm{var}}(r,d)$;

\item [$(ii)$] \ \ $E_{\mathcal{G}}(\lambda)\le E_{\cH_d^r}(\lambda)$ for every $0<\lambda\le\lambda_{\mathrm{occ}}(r,d)$.
\end{itemize}
\end{corollary}


We first prove the following lemma needed for this corollary.

\begin{lemma}
\label{lem:H-expansion}
For $r\ge3$ and $d\ge2$, we have
\begin{equation*}\label{eq:H-g-expansion}
g_{\cH_d^r}(p)
=-\frac drp^r+\frac{d(d-1)}{2r}p^{r+1}
+O_{r,d}(p^{r+2}).
\end{equation*}
\end{lemma}

\begin{proof}
By Lemma~\ref{lem:normalization} and $|V(\cH_d^r)|=rd$, we have
$Q_{\cH_d^r}(p) =(1-p)^{rd}Z_{\cH_d^r}\left(\frac{p}{1-p}\right)$.
By direct calculation, we have
$ Z_{\cH_d^r}(\lambda) = (1+\lambda)^{(r-1)d} +\bigl((1+\lambda)^d-1\bigr) \bigl((1+\lambda)^{r-1}-\lambda^{r-1}\bigr)^d.$
Using $\lambda=\frac{p}{1-p}$ and $1+\lambda=\frac1{1-p}$, we have
\begin{align*}
Q_{\cH_d^r}(p) =& (1-p)^{rd} \Big((1+\lambda)^{(r-1)d} +\bigl((1+\lambda)^d-1\bigr) \bigl((1+\lambda)^{r-1}-\lambda^{r-1}\bigr)^d\Big) \notag\\
=& (1-p)^{rd} \Bigg( (1-p)^{-(r-1)d} + ((1-p)^{-d}-1)((1-p)^{-(r-1)}-\big(\frac{p}{1-p}\big)^{r-1})^d \Bigg) \notag\\
=& (1-p)^{d} + (1-p)^{rd} \big((1-p)^{-d}-1 \big) (1-p^{r-1})^d (1-p)^{-(r-1)d}  \notag\\
=& (1-p)^{d} + (1-p)^{d}(1-p^{r-1})^d \big((1-p)^{-d}-1 \big) \notag\\
=& (1-p)^{d} + (1-p^{r-1})^d \Big( 1-(1-p)^d \Big) .
\end{align*}
For brevity, let $ A(p)=1-(1-p)^d$, $B(p)=(1-p^{r-1})^d$.
Then
$
Q_{\cH_d^r}(p)=1+A(p)(B(p)-1)$.

Note that
$A(p)=dp-\binom d2 p^2+O_d(p^3)$
 and
$B(p)-1 =-dp^{r-1}+O_{r,d}(p^{2r-2})$
where the definition of $O_d$ is similar to that of $O_{r,d}$.
Then
\begin{align*}
A(p)(B(p)-1)
&=\bigl(dp-\binom d2 p^2+O_d(p^3)\bigr) \bigl(-dp^{r-1}+O_{r,d}(p^{2r-2})\bigr)\\
&=-d^2p^r+d\binom d2 p^{r+1} +O_{r,d}(p^{r+2})+O_{r,d}(p^{2r-1}).
\end{align*}
Since $r\ge3$ implies $2r-1\ge r+2$, the last two error terms can be combined.
Thus
\begin{equation}\label{eq:QH-low-order}
Q_{\cH_d^r}(p)
=1-d^2p^r+d\binom d2 p^{r+1}+O_{r,d}(p^{r+2}).
\end{equation}

Let $T(p) := Q_{\cH_d^r}(p)-1 $. Then $T(p)=-d^2p^r+d\binom{d}{2} p^{r+1} + O_{r,d}(p^{r+2})$
from \eqref{eq:QH-low-order}.
Note that $T(p)=O_{r,d}(p^r)\to0$ as $p\to0^+$.
And $Q_{\mathcal H_d^r}(p)=1+T(p)>0$ for all sufficiently small $p\ge0$.
The Taylor expansion $\log(1+t)=t+O(t^2)$ as $t\to0$ gives
$ \log (1+T(p)) =T(p)+O(T(p)^2) =T(p)+O_{r,d}(p^{2r}). $
Consequently, we have
\begin{align*}
\log Q_{\cH_d^r}(p) &= T(p)+O(T^2) \\
&= -d^2p^r+d\binom{d}{2} p^{r+1} + O_{r,d}(p^{r+2}) + O(p^{2r}) \\
&= -d^2p^r+d\binom{d}{2} p^{r+1} + O_{r,d}(p^{r+2}),
\end{align*}
which implies that
\begin{align*}
g_{\mathcal{H}_d^r}(p)= \frac{1}{|V(\mathcal{H}_d^r)|} \log Q_{\mathcal{H}_d^r}(p)
= -\frac drp^r+\frac{d(d-1)}{2r}p^{r+1} +O_{r,d}(p^{r+2}).
\end{align*}
So the proof  is completed.
\end{proof}

\begin{proof}[Proof of Corollary~\ref{cor:variance-comparison}]
Let $0<p\le\zeta/2$ and $t=p/\zeta$.
Since $0<\zeta<1$, we have $0<p<1/2$ and $0<t\le1/2$.
By Lemma~\ref{lem:linear-expansion}, the coefficient of $p^{2r-1}$ in $g_{\mathcal G}(p)$ is $\binom d2$.
Recall that $M=\frac{3d}{20rA_0}$.
Proposition~\ref{prop:coefficient-bound} therefore implies
\begin{align}\label{d2rho}
\binom d2\zeta^{2r-1} = |b_{2r-1}(\mathcal G)|\zeta^{2r-1}
\le \sum_{k\ge1}|b_k(\mathcal G)|\zeta^k \le M.
\end{align}

By Claim \ref{kxk} in the proof of Lemma \ref{lem:weighted-derivative-tail}, we have the following inequalities:
\begin{align}\label{monotonicity}
kt^k\le(r+2)t^{r+2},
\qquad
k^2t^k\le(r+2)^2t^{r+2}
\qquad(k\ge r+2).
\end{align}

Since $r\ge3$, we have $2r-1\ge r+2$. So these inequalities apply to $k\in\{2r-1,2r+1\}$.
By Lemma~\ref{lem:H-expansion}, the function $g_{\mathcal H_d^r}(p)$ can be written as
$ g_{\mathcal H_d^r}(p) = -\frac drp^r +\frac{d(d-1)}{2r}p^{r+1} +R_{\mathcal H_d^r}(p)$,
where $R_{\mathcal H_d^r}(p)=\sum_{k\ge r+2}b_k(\mathcal H_d^r)p^k$. By
Proposition~\ref{prop:coefficient-bound}, we have
$
\sum_{k\ge r+2}|b_k(\mathcal H_d^r)|\zeta^k\le M$.
Applying Lemma~\ref{lem:weighted-derivative-tail} with $m=r+2$ and $\omega=M$ yields
\begin{align}
|\mathcal D R_{\mathcal H_d^r}(p)| &\le M(r+2)(1-p)t^{r+2}, \label{eq:DRH}\\
|\mathcal D^2 R_{\mathcal H_d^r}(p)| &\le M(r+2)^2(1-p)^2t^{r+2}. \label{eq:D2RH}
\end{align}

By \eqref{d2rho} and \eqref{monotonicity}, we obtain
\begin{align}\label{eq:cof_2r-1}
(2r-1)\binom d2(1-p)p^{2r-1} &= (1-p)\left(\binom d2\zeta^{2r-1}\right) (2r-1)t^{2r-1} \notag \\
&\le M(1-p)(2r-1)t^{2r-1} \notag \\
&\le M(r+2)(1-p)t^{r+2}.
\end{align}
Similarly, by \eqref{monotonicity}, we have
\begin{align}\label{k=2r+1}
(2r+1)M(1-p)t^{2r+1} \le M(r+2)(1-p)t^{r+2}.
\end{align}

Recall that $\theta_E(p) =p-d(1-p)p^r +(2r-1)\binom d2(1-p)p^{2r-1} -dA_0(1-p)p^{2r}$ where
the last term $-dA_0(1-p)p^{2r}$ is nonpositive.
Thus, by \eqref{eq:cof_2r-1} and \eqref{k=2r+1}, we have
\begin{align*}
\theta_E(p)+(2r+1)M(1-p)t^{2r+1} \le p-d(1-p)p^r +2M(r+2)(1-p)t^{r+2}.
\end{align*}

By \eqref{eq:F-normalized}, \eqref{eq:refined-monomial-derivatives} and \eqref{eq:DRH}, we have
\begin{align}
E_{\mathcal H_d^r}(\lambda)
=&p+\mathcal Dg_{\mathcal H_d^r}(p)  \notag\\
=&p-d(1-p)p^r +\frac{d(d-1)(r+1)}{2r}(1-p)p^{r+1} +\mathcal D R_{\mathcal H_d^r}(p) \notag\\
\ge&p-d(1-p)p^r +\frac{d(d-1)(r+1)}{2r}(1-p)p^{r+1} -M(r+2)(1-p)t^{r+2}. \label{lowEH}
\end{align}

Next we consider the difference between the value of $E_{\mathcal H_d^r}(\lambda)$ and the upper bound on $E_\mathcal G(\lambda)$.
By \eqref{EG} and \eqref{lowEH}, we get
\begin{align}
&E_{\mathcal H_d^r}(\lambda)-\theta_E(p)
-(2r+1)M(1-p)t^{2r+1}\notag\\
 \ge& \frac{d(d-1)(r+1)}{2r}(1-p)p^{r+1} -M(r+2)(1-p)t^{r+2}\notag\\
\qquad &  -(2r-1)\binom d2(1-p)p^{2r-1} +dA_0(1-p)p^{2r} -(2r+1)M(1-p)t^{2r+1} \notag\\
 \ge & \frac{d(d-1)(r+1)}{2r}(1-p)p^{r+1} -3M(r+2)(1-p)t^{r+2}\notag\\
=& (1-p)p^{r+1} \left( \frac{d(d-1)(r+1)}{2r} -\frac{3M(r+2)}{\zeta^{r+2}}p \right). \label{eq:occupancy-comparison-gap}
\end{align}
The second inequality is obtained from \eqref{eq:cof_2r-1}, \eqref{k=2r+1} and $dA_0(1-p)p^{2r}\ge0$.

If $p\le p_{\mathrm{occ}}(r,d)$, then $p\le\widehat p_{\mathrm{occ}}(r,d)$ by definition.
Using \eqref{eq:comparison-pocc-hat} and $M=3d/(20rA_0)$, we have
\begin{align*}
\frac{d(d-1)(r+1)}{2r} -\frac{3M(r+2)}{\zeta^{r+2}}p
\ge& \frac{d(d-1)(r+1)}{2r} -\frac{3M(r+2)}{\zeta^{r+2}} \widehat p_{\mathrm{occ}}(r,d)\\
=&
\frac{d(d-1)(r+1)}{2r} -\frac{3d(d-1)(r+1)(r+2)}{8r(r+3)}\\
=&
\frac{d(d-1)(r+1)(r+6)}{8r(r+3)} >0.
\end{align*}
Since $(1-p)p^{r+1}>0$, $E_{\mathcal H_d^r}(\lambda)\ge \theta_E(p)+(2r+1)M(1-p)t^{2r+1}$ whenever $0<p\le p_{\mathrm{occ}}(r,d)$.

For the variance fraction, we additionally assume that $p\le1/8$, as ensured by the definition of $p_{\mathrm{var}}(r,d)$.
Recall that
\begin{align*}
\theta_V(p)
=&p(1-p)-d(1-p)\bigl(r-(r+1)p\bigr)p^r +(2r-1)\binom d2(1-p)\bigl(2r-1-2rp\bigr)p^{2r-1}\\
&-dA_0(1-p)\bigl(2r-(2r+1)p\bigr)p^{2r}.
\end{align*}

Since $p<1/2$, $0<2r-1-2rp=(2r-1)(1-p)-p\le(2r-1)(1-p).$
Combining this inequality with \eqref{d2rho} gives
\begin{align*}
(2r-1)\binom d2(1-p) \bigl(2r-1-2rp\bigr)p^{2r-1} &\le (2r-1)^2\binom d2(1-p)^2p^{2r-1}\\
&= (1-p)^2\left(\binom d2\zeta^{2r-1}\right) (2r-1)^2t^{2r-1}\\
&\le M(1-p)^2 (2r-1)^2t^{2r-1}\\
&\le M(1-p)^2 (r+2)^2t^{r+2}.
\end{align*}
Similarly, we have
$ (2r+1)^2(1-p)^2t^{2r+1}M \le (r+2)^2(1-p)^2t^{r+2}M. $
The last term of $\theta_V(p)$ is nonpositive because $2r-(2r+1)p>0$.
Therefore, we have
\begin{align}\label{eq:theta-V+re}
&\theta_V(p)+(2r+1)^2M(1-p)^2t^{2r+1} \notag\\ \le& p(1-p)-d(1-p)\bigl(r-(r+1)p\bigr)p^r + 2M(r+2)^2(1-p)^2t^{r+2}.
\end{align}

Note that
\begin{align}\label{eq:1}
(1-p)\bigl(r+1-(r+2)p\bigr)
&=r+1-(2r+3)p+(r+2)p^2 \notag\\
&\ge r+1-\frac{2r+3}{8} \ge\frac{r+1}{2}.
\end{align}

Applying \eqref{eq:refined-monomial-derivatives} to $V_{\mathcal H_d^r}(\lambda)$ and using \eqref{monotonicity} and \eqref{eq:1}, we obtain
\begin{align}
V_{\mathcal H_d^r}(\lambda)
&=p(1-p)-d(1-p)\bigl(r-(r+1)p\bigr)p^r \notag\\
&\quad+\frac{d(d-1)(r+1)}{2r} (1-p)\bigl(r+1-(r+2)p\bigr)p^{r+1} + \mathcal D^2R_{\mathcal H_d^r}(p)\notag\\
&\ge p(1-p)-d(1-p)\bigl(r-(r+1)p\bigr)p^r +\frac{d(d-1)(r+1)^2}{4r}p^{r+1} \notag\\
&\quad -M(r+2)^2(1-p)^2t^{r+2}. \label{VH}
\end{align}

Combining with \eqref{eq:theta-V+re}, \eqref{eq:1} and \eqref{VH}, we have
\begin{align}
&V_{\mathcal H_d^r}(\lambda)-\theta_V(p) -(2r+1)^2M(1-p)^2t^{2r+1}\notag\\
\ge& \frac{d(d-1)(r+1)^2}{4r}p^{r+1} -3M(r+2)^2(1-p)^2t^{r+2}\notag\\
\ge& \frac{d(d-1)(r+1)^2}{4r}p^{r+1} -3M(r+2)^2t^{r+2}\notag\\
=& p^{r+1} \left( \frac{d(d-1)(r+1)^2}{4r} -\frac{3M(r+2)^2}{\zeta^{r+2}}p \right). \label{eq:variance-comparison-gap}
\end{align}

If $p\le p_{\mathrm{var}}(r,d)$, then $p\le\widehat p_{\mathrm{var}}(r,d)$.
Using \eqref{eq:comparison-pvar-hat} and $M=3d/(20rA_0)$, we obtain
\begin{align*}
&\frac{d(d-1)(r+1)^2}{4r} -\frac{3M(r+2)^2}{\zeta^{r+2}}p\\
\ge& \frac{d(d-1)(r+1)^2}{4r} -\frac{3M(r+2)^2}{\zeta^{r+2}} \widehat p_{\mathrm{var}}(r,d)\\
=& \frac{d(d-1)(r+1)^2}{4r}-\frac{d(d-1)(r+1)^2}{192r}\\
=& \frac{47d(d-1)(r+1)^2}{192r} >0.
\end{align*}
So we have $V_{\mathcal H_d^r}(\lambda)>\theta_V(p)+(2r+1)^2M(1-p)^2t^{2r+1}$ whenever $0<p\le p_{\mathrm{var}}(r,d)$.

By definitions, $p_{\mathrm{var}}(r,d)\le p_{\mathrm{occ}}(r,d)\le\zeta/2$.
Since $p/(1-p)$ is increasing on $[0,1)$, we have
$$
\begin{aligned}
\lambda_{\mathrm{var}}(r,d)
&=\frac{p_{\mathrm{var}}(r,d)}{1-p_{\mathrm{var}}(r,d)}
\le\frac{p_{\mathrm{occ}}(r,d)}{1-p_{\mathrm{occ}}(r,d)}
=\lambda_{\mathrm{occ}}(r,d),\\
\lambda_{\mathrm{occ}}(r,d)
&\le\frac{\zeta/2}{1-\zeta/2}
=\frac{\zeta}{2-\zeta}
=\lambda_0(r,d).
\end{aligned}
$$
Combining with the increasing relation $\lambda=p/(1-p)$, \eqref{EG} and \eqref{VG}, we prove the result.
\end{proof}

\section*{Acknowledgments}
This work is partially supported by National Natural Science Foundation of China (No.\ 12271251) and Postgraduate Research \& Practice Innovation Program of Jiangsu Province (No.\ 26CXJH0986).

\section*{Author Contributions} All authors contributed equally to this work.
\section*{Data Availability} No data was used in this article
\section*{Author Contributions} The authors declare that they have no Conflict of interest.

\section*{Declarations}
During the preparation of this work, the authors used AI systems to assist in exploring potential approaches to the proof of the main theorems. All AI-generated suggestions were verified and refined by the authors.


\begin{thebibliography}{99}

\bibitem{A1996} N. Alon, Independence numbers of locally sparse graphs and a Ramsey type problem, Random Struct. Algorithms 9 (1996) 271--278.

\bibitem{BBN2021} J. Balogh, B. Bollob\'as, B. Narayanan, Counting independent sets in regular hypergraphs, J. Combin. Theory Ser. A, 180 (2021) 105405.

\bibitem{BMS2015} J. Balogh, R. Morris, W. Samotij, Independent sets in hypergraphs, J. Amer. Math. Soc. 28 (2015) 669--709.

\bibitem{BencsBuys2025} F. Bencs, P. Buys, Optimal zero-free regions for the independence polynomial of bounded degree hypergraphs, Random Struct. Algorithms 66 (2025) e70018.


\bibitem{BLM2022} T. Bohman, X. Liu, D. Mubayi, Independent sets in hypergraphs omitting an intersection, Random Struct. Algorithms 61 (2022) 493--519.


\bibitem{BCKL2013} C. Borgs, J. Chayes, J. Kahn, L. Lov\'{a}sz, Left and right convergence of graphs with bounded degree, Random Struct. Algorithms 42 (2013) 1--28.

\bibitem{CPST2022} E. Cohen, W. Perkins, M. Sarantis, P. Tetali, On the number of independent sets in uniform, regular, linear hypergraphs, European J. Combin. 99 (2022) 103401.

\bibitem{CDM2014} J. Cooper, K. Dutta, D. Mubayi, Counting independent sets in hypergraphs, Combin. Probab. Comput. 23 (2014) 539--550.

\bibitem{CR2013} J. Cutler, A. J. Radcliffe, Hypergraph independent sets, Combin. Probab. Comput. 22 (2013) 9--20.


\bibitem{CR2014} J. Cutler, A. Radcliffe, The maximum number of complete subgraphs in a graph with given maximum degree, J. Combin. Theory Ser. B 104 (2014) 60--71.

\bibitem{CR2018} J. Cutler, A. Radcliffe, Minimizing the number of independent sets in triangle-free regular graphs, Discrete Math. 341 (2018) 793--800.

\bibitem{DVKP2020} E. Davies, R. de Joannis de Verclos, R. Kang, F. Pirot, Coloring triangle-free graphs with local list sizes, Random Struct. Algorithms 57 (2020) 730--744.

\bibitem{DVKP2021} E. Davies, R. de Joannis de Verclos, R. Kang, F. Pirot, Occupancy fraction, fractional colouring, and triangle fraction, J. Graph Theory 97 (2021) 557--568.

\bibitem{DJP2021} E. Davies, M. Jenssen, W. Perkins, A proof of the upper matching conjecture for large graphs, J. Combin. Theory Ser. B 151 (2021) 393--416.

\bibitem{DJPR2017} E. Davies, M. Jenssen, W. Perkins, B. Roberts, Independent sets, matchings, and occupancy fractions, J. Lond. Math. Soc. 96 (2017) 47--66.

\bibitem{DJPR2018} E. Davies, M. Jenssen, W. Perkins, B. Roberts, On the average size of independent sets in triangle-free graphs, Proc. Amer. Math. Soc. 146 (2018) 111--124.

\bibitem{DP2023} E. Davies, W. Perkins, Approximately counting independent sets of a given size in bounded-degree graphs, SIAM J. Comput. 52 (2023) 618--640.

\bibitem{DSST2026} E. Davies, J. Sandhu, J. Seo, B. Tan, Degree-sequence bounds for independent sets via multivariate local occupancy, (2026) arXiv:2605.05149.

\bibitem{DST2025} E. Davies, J. Sandhu, B. Tan, On expectations and variances in the hard-core model on bounded degree graphs, (2025) arXiv:2505.13396.

\bibitem{FP2007} R. Fern\'andez, A. Procacci, Cluster expansion for abstract polymer models: new bounds from an old approach, Comm. Math. Phys. 274 (2007), 123--140.

\bibitem{FKM2008} S. Friedland, E. Krop, K. Markstr\"{o}m, On the number of matchings in regular graphs, Electron. J. Comb. 15 (1) (2008) R110.


\bibitem{GK1971} C. Gruber, H. Kunz, General properties of polymer systems, Comm. Math. Phys. 22 (1971) 133--161.



\bibitem{JP2020} M. Jenssen, W. Perkins, Independent sets in the hypercube revisited, J. Lond. Math. Soc. 102 (2020), 645--669.

\bibitem{JPP2022} M. Jenssen, W. Perkins, A. Potukuchi, Independent sets of a given size and structure in the hypercube, Combin. Probab. Comput. 31 (2022) 702--720.

\bibitem{JPP2023} M. Jenssen, W. Perkins, A. Potukuchi, Approximately counting independent sets in bipartite graphs via graph containers, Random Struct. Algorithms 63 (2023) 215--241.


\bibitem{JT2020}
S. Jansen, D. Tsagkarogiannis,
Cluster expansions with renormalized activities and applications to colloids,
Ann. Henri Poincar\'e 21 (2020) 45--79.

\bibitem{Jukna2011} S. Jukna, \emph{Extremal Combinatorics: With Applications in Computer Science}, Springer, Berlin, 2011.

\bibitem{K2001} J. Kahn, An entropy approach to the hard-core model on bipartite graphs, Combin. Probab. Comput. 10 (2001) 219--237.

\bibitem{KP1986} R. Koteck\'y, D. Preiss, Cluster expansion for abstract polymer models, Comm. Math. Phys. 103 (1986) 491--498.


\bibitem{LPV1967}
J. Lebowitz, J. Percus, L. Verlet,
Ensemble dependence of fluctuations with application to machine computations,
Phys. Rev. 153 (1967) 250--254.

\bibitem{MS2025} A. Martinsson, R. Steiner,
Random independent sets in triangle-free graphs,
Forum Math. Sigma 13 (2025) 1--19.

\bibitem{M2019} M. Molloy, The list chromatic number of graphs with small clique number, J. Combin. Theory Ser. B 134 (2019) 264--284.




\bibitem{PP2018} G. Perarnau, W. Perkins, Counting independent sets in cubic graphs of given girth, J. Combin. Theory Ser. B 133 (2018) 211--242.



\bibitem{PT2012} E. Pulvirenti, D. Tsagkarogiannis,
Cluster expansion in the canonical ensemble,
Comm. Math. Phys. 316 (2012) 289--306.

\bibitem{PT2015} E. Pulvirenti, D. Tsagkarogiannis,
Finite volume corrections and decay of correlations in the canonical ensemble,
J. Stat. Phys. 159 (2015) 1017--1039.


\bibitem{PU2009}
S. Poghosyan, D. Ueltschi, Abstract cluster expansion with applications to statistical mechanical systems,
J. Math. Phys. 50 (2009) 053509.


\bibitem{SSSZ2019} A. Sah, M. Sawhney, D. Stoner, Y. Zhao, The number of independent sets in an irregular graph, J. Combin. Theory Ser. B 138 (2019) 172--195.

\bibitem{ST2015} D. Saxton, A. Thomason, Hypergraph containers, Invent. Math. 201 (2015) 925--992.

\bibitem{SS2005} A. Scott, A. Sokal, The repulsive lattice gas, the independent-set polynomial, and the Lov\'asz local lemma, J. Stat. Phys. 118 (2005) 1151--1261.


\bibitem{S1995} J. Shearer, On the independence number of sparse graphs, Random Struct. Algorithms 7 (1995) 269--271.


\bibitem{WET2024}
N. Wilding, R. Evans, F. Turci,
What is the best simulation approach for measuring local density fluctuations near solvo-/hydrophobes?
J. Chem. Phys. 160 (2024) 164103.


\bibitem{ZX2026} W. Zhang, K. Xu, On the variance fraction of the hard-core model on graphs with bounded maximum degree, (2026) arXiv:2604.01717v3.

\bibitem{Z2010} Y. Zhao, The number of independent sets in a regular graph, Combin. Probab. Comput. 19 (2010) 315--320.

\bibitem{Z2017} Y. Zhao, Extremal regular graphs: independent sets and graph homomorphisms, Amer. Math. Monthly 124 (2017) 827--843.

\end{thebibliography}
\end{document}